\documentclass[12pt]{article}

\usepackage{datetime}
\usepackage{arxiv}

\usepackage[utf8]{inputenc} 
\usepackage[T1]{fontenc}    
\usepackage[colorlinks=true, allcolors=black]{hyperref}       
\usepackage{url}            
\usepackage{booktabs}       
\usepackage{amsfonts}       
\usepackage{nicefrac}       
\usepackage{microtype}      
\usepackage{lipsum}
\usepackage{graphicx}
\usepackage{appendix}
\usepackage{amsmath}
\graphicspath{ {./images/} }
\usepackage{mwe}
\usepackage{subfig}
\usepackage{amsthm}

\usepackage{scalerel,stackengine}

\stackMath
\newcommand\reallywidehat[1]{%
\savestack{\tmpbox}{\stretchto{%
  \scaleto{%
    \scalerel*[\widthof{\ensuremath{#1}}]{\kern-.6pt\bigwedge\kern-.6pt}%
    {\rule[-\textheight/2]{1ex}{\textheight}}
  }{\textheight}%
}{0.75ex}}%
\stackon[1pt]{#1}{\tmpbox}%
}
\theoremstyle{plain}

\newtheorem{theorem}{Theorem}[section]
\newtheorem{proposition}{Proposition}[section]

\theoremstyle{definition}
\newtheorem{definition}{Definition}[section]

\newcommand{\su}[3]{\displaystyle \sum_{{#1}}^{{#2}}{#3}}
\newcommand{\itg}[3]{\displaystyle \int_{{#1}}^{{#2}}{#3}}

\title{Mitigation of Artifacts in Multistatic \& Passive Radar Imaging Using Microlocal Analysis}

\author{
 David McMahon \\
  Department of Mathematics \& Statistics\\
  University of Limerick\\
  \texttt{david.j.mcmahon@ul.ie} \\
   \And
 Clifford Nolan \\
  Department of Mathematics \& Statistics\\
  University of Limerick\\
  \texttt{clifford.nolan@ul.ie} \\
}

\begin{document}
\maketitle
\begin{abstract}
In the analysis of many synthetic aperture radar (SAR) experiments the possibility of passive background signals being recorded simultaneously and corrupting the image is often overlooked. Our work addresses this by considering the multistatic experiment where two stationary emitters are `always on’ so there is `crosstalk' between their signals. The model for the radar data is given by a Fourier integral operator, and we assume that the data cannot be separated into contributions from individual emitters. Using techniques of microlocal analysis, we show that `crosstalk' between emitters leads to artifacts in the image and we determine their locations relative to the scatterers that produced the data.

 To combat the harmful effects of crosstalk, we develop methods that allow us to create an image of a region of interest (ROI) that is free from such artifacts. The first method makes use of a carefully designed data acquisition geometry to localise artifacts away from a ROI, and the second is an image processing technique that displaces artifacts away from a ROI. These methods are verified via numerical implementation in MATLAB. The analysis carried out here is valuable in bistatic and multistatic radar experiments, where an unwanted, passive source is also being detected, as well as in passive imaging, where one wishes to produce a high-quality image purely from uncontrolled sources of illumination.

\end{abstract}


\section{Introduction}
Synthetic aperture radar (SAR) \cite{SARinversion, tutorial, nonflat, nolan2004microlocal, fundamentals} is a high-resolution imaging technique used to create images of objects or environments remotely. SAR uses antennas on moving platforms (usually a plane or satellite) to send electromagnetic waves to objects of interest and measure the scattered waves. These measurements are called range profiles, and they associate the two-way travel time, $2t$, of the backscattered waves with objects that are at a distance of $c_0 t$ from the antenna (where $c_0$ is the speed of light). Radar data is comprised of range profiles recorded at each point along the flight track. The data is used to produce a reconstruction of the terrain that resulted in the recorded range profiles. In monostatic SAR imaging, the moving transmitter also acts as a receiver, whereas in bistatic SAR imaging, the transmitter and receiver are located on different platforms. Multistatic radar refers to a setup in which there are multiple monostatic or bistatic radar components with a shared area of coverage.

The novel aspect of this work is that we address the potential for crosstalk between various emitters in radar experiments and mitigate this effect when reconstructing an image of a scene. We do this by considering a multistatic SAR setup comprised of two stationary sources of illumination that are measured by a single moving receiver. Crosstalk is often neglected as it is typically assumed that only the signal from one's own transceiver is measured during an experiment \cite{SARinversion, nonflat}. In practice, however, there can be signals from various background sources that will also be recorded. This is especially true when imaging in urban environments, where there is a higher density of background signals present (such as those from radio towers, for example). We assume that the emitters are active simultaneously and that the data cannot be separated into the contributions measured from each. The receiver makes multiple passes over the scattering region, and as such, we expect to create a three-dimensional reconstruction of the scattering region \cite{SARinversion}. This experiment is relevant in contexts such as bistatic radar, where an uncontrolled signal is recorded during the experiment and corrupts the data. Our analysis is also applicable to passive radar imaging, where only background signals are used to produce an image. Throughout this article, it is assumed that the emitter locations are known, which may not always be the case, particularly in passive imaging experiments. However, in such a scenario, our analysis can be used in conjunction with source localisation methods, such as in \cite{Waddington_2020}, to determine the emitter locations.

Most of the important features in imaging are the points and edges between different media where scattering occurs. From a mathematical perspective, such features are modelled as a type of singularity. These singularities are encoded in the reflectivity function $V$, which captures non-smooth changes in the index of refraction of the material in which the wave is propagating. We will derive a scattering operator $F$ that maps $V$ to the scattered waves, $FV$, recorded at the various receiver locations and times. From a mathematical perspective, the SAR imaging task is to recover $V$ from the measured data, $FV$. Since exact reconstruction of $V$ from $FV$ is often extremely difficult, an acceptable compromise in most applications is to instead reconstruct the singularities of $V$. This yields an image containing the edges and shapes of the objects that are present in the scattering region.  Microlocal analysis \cite{Duistermaat2011, grigis1994microlocal, TrevesJean-François1980ItPa, hormander2009analysis, krishnan2015microlocal, greenleaf.uhlmann.90a} is a mathematical theory that describes how singularities are propagated by a certain class of operators known as Fourier integral operators (FIOs). Microlocal analysis allows us to associate singularities in the radar data ($FV$) with the singularities in the scene ($V$) that produced the data. A reconstruction can be recreated through a process called backprojection, which attempts to replace the singularities in the data in their correct location in the scene.

In practice, backprojection may migrate a singularity to an  incorrect location, resulting in the appearance of an object in the image that is not present in the scene. The appearance of such a non-existent object is known as an artifact \cite{Felea01112005, Nolan01011997, plamen, 1512397}. As we are assuming that the data cannot be separated into the contributions measured from each emitter, it will turn out that we are not able to choose a backprojection operator that correctly reconstructs the singularities from  each part of the data simultaneously. This implies that there will be artifacts caused by crosstalk in the reconstruction. We determine the locations of such artifacts in the image as a result of a known point scatterer and show that the artifacts extend over a surface parameterised by the location of the receiver. 

Using the knowledge that we have gained about these artifacts, we can devise two methods that allow us to image a region of interest (ROI) that is completely free from artifacts. The first method makes use of the fact that the extent of the surface containing artifacts depends on the receiver location. By restricting the flight track appropriately (either by design of the data acquisition geometry or by omitting certain portions of data before backprojection), it is possible to restrict the surface of artifacts such that they lie strictly outside the ROI. Our second method applies the work carried out in \cite{Felea_2007} to our multistatic experiment. This allows us to process the reconstruction by applying a series of FIOs that progressively displace artifacts further away from the ROI. We implement these methods numerically in MATLAB, and both are observed to significantly reduce the presence of crosstalk artifacts in the reconstructions.

The remainder of this paper is organised as follows. In Section \ref{forward_model_section},
we describe the model for scattered waves resulting from two stationary emitters at known locations. Using a single scattering approximation we see that this forward model is given by a sum of FIOs, each corresponding to the contribution received from each emitter. In Section \ref{surface_section}, we apply a backprojection operator and analyse the resulting reconstructions. Application of the backprojection operator yields two types of terms. The firstis a pseudodifferential operator that provides the correct reconstruction of singularities. The second type of term is the result of incorrectly perceiving the signal from one emitter as having come from the other emitter. As such, this term models the impact of crosstalk that we wish to investigate. We show that this term is an FIO, and we gain valuable insights into the artifacts it causes in the reconstruction.

In Section \ref{mitigation}, we develop two basic imaging algorithms designed to effectively remove all artifacts from a ROI. As already mentioned, these methods can carefully restrict the data acquisition geometry so that artifacts are outside a ROI; otherwise, we can perform a sequence of image processing steps that iteratively displace artifacts away from a ROI. We discuss the advantages of each method, the conditions under which they are appropriate to use each (particularly the artifact displacement method) and give attention to drawbacks they might have or when one method might be preferable over the other. In Section \ref{numerics}, we numerically illustrate and validate the theoretical results discussed throughout this article. We first simulate a SAR experiment in which `crosstalk' is present and highlight how detrimental this can be on the quality of the reconstruction. From here, each of the two methods, outlined above, are implemented to remove these artifacts from the ROI. In each case, a substantial improvement in the quality of the image is observed, thereby displaying their effectiveness.

\section{Forward model}\label{forward_model_section}
We begin by constructing the mathematical model which describes the measurements, made by a moving receiver, of the electromagnetic field resultant from two stationary emitters. The basis for this model is the scalar wave equation,

\begin{equation*}
	\bigg(\nabla^2 -\frac{1}{c^2(\mathbf{x})}\partial^2_t\bigg)U(t, \mathbf{x}) = 0,
\end{equation*}

where $\mathbf{x} \in \mathbb{R}^3$ is position in space and $t \in  \mathbb{R}$ is the time variable. Here $U$ represents one component of the electromagnetic field. 

We assume that the scene to be imaged is well separated from the region where the data is recorded, and that $c(\mathbf{x}) \equiv c_0$, in the intervening region, where $c_0$ is the speed of light in air, assumed to be constant for simplicity.

To construct the solution for $U(t, \mathbf{x})$,  
 we use the Green's function for the scalar wave equation,
 
 \begin{equation*}
 	G_0(t,\mathbf{x}) = \frac{\delta(t-\vert \mathbf{x}\vert/c_0)}{4\pi\vert \mathbf{x} \vert}\ .
 \end{equation*}

For a pair of stationary point sources, located at $\mathbf{E}_1$ and $\mathbf{E}_2$ respectively, the resultant incident field, $U^{\text{in}}$, satisfies
\begin{equation}
	\bigg(\nabla^2 -\frac{1}{c_0^2(\mathbf{x})}\partial^2_t\bigg)U^{\text{in}}(t, \mathbf{x}) = -\left(\delta(t)\delta(\mathbf{x}-\mathbf{E}_1) + \delta(t)\delta(\mathbf{x}-\mathbf{E}_2)\right),\label{incident_field}
\end{equation}

and $U^\text{in}(t, \mathbf{x}) \equiv 0$ for $t < 0$.

Using the Green's function, we obtain

\begin{equation*}
	U^{\text{in}}(t,\mathbf{x}) 
    = \frac{\delta(t - |\mathbf{x}-\mathbf{E}_1|/c_0)}{4\pi |\mathbf{x} - \mathbf{E}_1|}  +  \frac{\delta(t - |\mathbf{x}-\mathbf{E}_2|/c_0)}{4\pi |\mathbf{x} - \mathbf{E}_2|} \ .
\end{equation*}

We recall the oscillatory integral form of the Dirac delta distribution,
    \begin{equation*}
        \delta (x) = \frac{1}{2\pi} \int_{\mathbb{R}} e^{-i\omega x}d\omega,
    \end{equation*}
and use it to write
\begin{equation}
	U^{\text{in}}(t,\mathbf{x}) = \int\left(\frac{e^{-i\omega (t - \vert \mathbf{x}- \mathbf{E}_1 \vert /c_0)}}{4\pi\vert \mathbf{x}-\mathbf{E}_1\vert} + \frac{e^{-i\omega (t - \vert \mathbf{x}- \mathbf{E}_2 \vert /c_0)}}{4\pi\vert \mathbf{x}-\mathbf{E}_2\vert}\right)d\omega .\label{U_in}
\end{equation}

\subsection{Linearisation}
Up to now we have concerned ourselves only with the field of electromagnetic waves that are incident from the antenna. That is the field before any reflection of waves from objects in the scene has occurred. The full field, $U$, is comprised of both the incident waves and the scattered waves, i.e., $U = U^{\text{in}} + U^{\text{sc}}$. The full field satisfies the wave equation,
\begin{equation}
    \bigg(\nabla^2 -\frac{1}{c^2(\mathbf{x})}\partial^2_t\bigg)U(t, \mathbf{x}) = -\left(\delta(t)\delta(\mathbf{x}-\mathbf{E}_1) + \delta(t)\delta(\mathbf{x}-\mathbf{E}_2)\right).\label{full_field}
\end{equation}

We use the fact that the right hand sides of (\ref{full_field}) and (\ref{incident_field}) are equal to set
\begin{equation*}
    \bigg(\nabla^2 -\frac{1}{c^2(\mathbf{x})}\partial^2_t\bigg)U(t, \mathbf{x}) = \bigg(\nabla^2 -\frac{1}{c_0^2(\mathbf{x})}\partial^2_t\bigg)U^{\text{in}}(t, \mathbf{x}).
\end{equation*}

Using $U^{\text{sc}}(t, \mathbf{x}) = U(t,\mathbf{x}) - U^{\text{sc}}(t, \mathbf{x})$ we rearrange the above to obtain
\begin{equation}
    \bigg(\nabla^2 -\frac{1}{c_0^2}\partial^2_t\bigg)U^{\text{sc}}(t, \mathbf{x}) = -V(\mathbf{x})\partial_t^2U(t,\mathbf{x}),\label{linearised}
\end{equation}

where $V(\mathbf{x}) = \frac{1}{c_0^2} -\frac{1}{c^2(\mathbf{x})}$. $V(\mathbf{x})$ is known as the reflectivity function and contains information about how the velocity of waves in a medium located at position $\mathbf{x}$ differs from the reference speed in free space. Singularities in $V(\mathbf{x})$ occur at locations $\mathbf{x}$ where the velocity of waves, $c(\mathbf{x})$, change rapidly; a concept we will presently make precise. This happens at the sharp boundaries between different media and it follows that the singularities of $V(\mathbf{x})$ are located at these same boundaries. By reconstructing the singularities of $V$, we  will form an image that contains edges between the different structures in the scattering region.

Applying the Green's function for the wave equation to equation (\ref{linearised}) gives
\begin{equation}
    U^{\text{sc}}(t, \mathbf{x}) = \int G_0(t - \tau, \mathbf{x} - \mathbf{z}) V(\mathbf{z}) \partial_\tau^2 U(\tau, \mathbf{z})d\tau d\mathbf{z}.  \label{scatter}
\end{equation}

We now apply the {\em Born Approximation} \cite{Herman, Langenberg}, which replaces the full field, $U$ on the right hand side of (\ref{scatter}), with the incident field, $U^{\text{in}}$, resulting in

\begin{equation*}
    U^{\text{sc}}(t, \mathbf{x}) \approx \int G_0(t-\tau, \mathbf{x} - \mathbf{z}) V(\mathbf{z}) \partial_\tau^2 U^{\text{in}} (\tau, \mathbf{z}) d\tau d\mathbf{z} .
\end{equation*}
This is a standard, but formal approximation, and is based on the assumption $V$ is small relative to $c_0^{-2}$. The approximation appears to be very robust and works very well in practice.

This approximation removes non-linearity from the forward problem by replacing two unknowns ($U$ and $V$) with one unknown, $V$, multiplied by the known incident field, $U^{\text{in}}(t, \mathbf{x})$.

Using the expression for the incident in \eqref{U_in}, and employing the Green's function again, we obtain
\begin{equation*}
    U^{\text{sc}}(\mathbf{r}, \mathbf{z}) \approx \int \left(\frac{e^{-i\omega(t-(\vert \mathbf{x}-\mathbf{z} \vert + \vert \mathbf{x} - \mathbf{E}_1 \vert)/c_0)}}{(4\pi)^2 \vert \mathbf{x} - \mathbf{E}_1 \vert} + \frac{e^{-i\omega(t-(\vert \mathbf{x}-\mathbf{z} \vert + \vert \mathbf{x} - \mathbf{E}_2 \vert)/c_0)}}{(4\pi)^2 \vert \mathbf{x} - \mathbf{E}_2 \vert}\right)\frac{\omega^2 V(\mathbf{x})}{\vert \mathbf{x} - \mathbf{z}\vert}d\omega d\mathbf{x}.
\end{equation*}

We consider our radar data as being composed of measurements made of the scattered field at receiver locations $\mathbf{z}=\gamma(\mathbf{r})$. We let the antenna centre be on a flight track, at a height $h$ above the ground, that is paramaterised by $\{\gamma (\mathbf{r})  = (r_1, r_2, h) \text{ }\vert\text{ } r_1^\text{min} < r_1 < r_1^\text{max}, \text{ } r_2^\text{min} < r_2 < r_2^\text{max}\}$. The abrupt edges of such a data acquisition geometry can lead to artifacts in the image. With this in mind we multiply the data by a smooth taper function or mute,  $m(r, t)$, whose support is a subset of the rectangle $[r_1^\text{min}, r_1^\text{max}] \times [r_2^\text{min}, r_2^\text{min}] \times [0, T]$. Therefore, we modify our model of the scattered wavefield to the following:

\begin{equation}
    U^{\text{sc}}(\mathbf{r}, t) \approx \int \left(\frac{e^{-i\omega(t-(\vert \mathbf{x}-\gamma(\mathbf{r}) \vert + \vert \mathbf{x} - \mathbf{E}_1 \vert)/c_0)}}{(4\pi)^2 \vert \mathbf{x} - \mathbf{E}_1 \vert} + \frac{e^{-i\omega(t-(\vert \mathbf{x}-\gamma(\mathbf{r}) \vert + \vert \mathbf{x} - \mathbf{E}_2 \vert)/c_0)}}{(4\pi)^2 \vert \mathbf{x} - \mathbf{E}_2 \vert}\right)\frac{\omega^2 m(\mathbf{r},t)}{\vert \mathbf{x}-\gamma(\mathbf{r})\vert}V(\mathbf{x})d\omega d\mathbf{x}. \label{signal}
\end{equation}

We denote the map from the distribution $V$ to data by the operator $F: \mathcal{E}'(\mathbb{R}^3) \rightarrow \mathcal{D}'(\mathbb{R}^2 \times \mathbb{R})$ where $\mathcal{D}'(X), \mathcal{E}'(X)$ respectively refer to distributions on a an open set (or manifold) $X$ and distributions with compact support on $X$. The operator, $F$, is given explicitly as

\begin{equation} 
		FV(\mathbf{r},t) = \su{i=1,2}{}{\itg{}{}{e^{-i\omega\left(t-\frac{\vert \mathbf{x}-\gamma(\mathbf{r}) \vert }{c_0} - \frac{\vert \mathbf{x} - \mathbf{E}_i \vert}{c_0}\right)}A_i(\mathbf{x}, \omega, \mathbf{r})V(\mathbf{x})d\omega d\mathbf{x}}},\label{forward}
\end{equation}

where
\begin{equation}
	A_i(\mathbf{x}, \omega, \mathbf{r}) = \frac{\omega^2m(\mathbf{r},t)}{(4\pi)^2\vert \mathbf{x} -\gamma(\mathbf{r})\vert \vert \mathbf{x} -\mathbf{E}_i\vert}.\label{amp}
\end{equation}

The amplitudes $A_i(\mathbf{x}, s, t, \omega)$, satisfy the condition that for any compact set $K$
\begin{equation*}
    \sup_{(\mathbf{r}, t, \mathbf{x}) \in K} \vert \partial^\alpha_\omega \partial^\beta_s \partial^\delta _t \partial^\rho _\mathbf{x} A_i(\mathbf{x}, \mathbf{r}, t, \omega) \vert \leq C_{K, \alpha, \beta, \delta, \rho }(1 + \omega^2)^{(2-\vert\alpha\vert)/2} \ .
\end{equation*}
The condition on the amplitudes together with equation $(\ref{forward})$ shows that the forward operator $F$ is a Fourier integral operator (FIO) \cite{Duistermaat2011}. In fact each summand in the definition of $F$ is also a (local) FIO.

\subsection{Analysis of forward operator}

\begin{definition}\cite{Duistermaat2011}
    If the Fourier Integral Operator, $T : \mathcal{E}'(X) \rightarrow \mathcal{E}'(Y)$, is given by the oscillatory integral
\begin{equation*}
    Tf(\mathbf{y}) = \int e^{i\phi(\mathbf{y}, \mathbf{x}, \omega)}a(\mathbf{y}, \mathbf{x}, \omega)f(\mathbf{x})d\omega d\mathbf{x}
\end{equation*}

then its (twisted) canonical relation is the set
\begin{equation*}
    \Lambda^\prime_T = \left\{ ( (\mathbf{y}, \boldsymbol{\eta}), (\mathbf{x}, \boldsymbol{\xi})): D_\omega \phi(\mathbf{y}, \mathbf{x}, \omega) = 0, \boldsymbol{\eta} = -D_y\phi(\mathbf{y}, \mathbf{x}, \omega), \boldsymbol{\xi} = D_x\phi(\mathbf{y}, \mathbf{x}, \omega),\ (\mathbf{y},\mathbf{x},\omega)\in ess\ supp (a) \right\},
\end{equation*}
\end{definition}
where $ess\ supp\ a$ refers to the essential support of the symbol $a$; see \cite{Duistermaat2011} for more detail on the required properties of the amplitude $a$ and phase $\phi$, for example.

Our data is the sum of locally defined FIOs, $FV = \sum_{i=1,2}F_iV$, where each $F_i$ has canonical relation, $\Lambda^\prime_i$, given by
\begin{align}
    \Lambda^\prime_i = \Big\{((\mathbf{r}, t, \boldsymbol{\rho}, \tau), (\mathbf{x}, \boldsymbol{\xi})): t = \frac{1}{c_0}\left(|\mathbf{x}-\gamma(\mathbf{r})| + |\mathbf{x}-\mathbf{E}_i|\right),\notag \\ \boldsymbol{\rho} =\frac{\tau}{c_0}\widehat{(\mathbf{x}-\gamma(\mathbf{r}))}_H, \boldsymbol{\xi} = -\frac{\tau}{c_0}\left( \widehat{\left(\mathbf{x} - \gamma(\mathbf{r})\right)} + \widehat{(\mathbf{x}-\mathbf{E}_i)} \right)\Big\}\ ,\label{canonical}
\end{align}
where $\widehat{(\mathbf{x}-\gamma(\mathbf{r}))}_H$ denotes the horizontal components of $\widehat{(\mathbf{x}-\gamma(\mathbf{r}))}$. For $F_i$ to be an FIO we require that $\boldsymbol{\xi} \neq 0$, so we assume that no scatterers lie on the line segments between $\gamma(\mathbf{r})$ and each $\mathbf{E}_i$, thereby ensuring $\boldsymbol{\xi} = \left( \widehat{\left(\mathbf{x} - \gamma(\mathbf{r})\right)} + \widehat{(\mathbf{x}-\mathbf{E}_i)} \right)\neq 0$.

We may verify that $F$ is an FIO by checking the following conditions: The phase function of each $F_i$ is $\phi_i(r, t, \mathbf{x}, \omega) = \omega\left(t - \frac{|\mathbf{x}-\gamma(\mathbf{r})|}{c_0} -\frac{|\mathbf{x}-\mathbf{E}_i|}{c_0}\right)$ is clearly homogeneous of degree 1 in $\omega$ and it is also clear that it is non-degenerate \cite{Duistermaat2011}.

The order of $F$ is calculated as follows
\begin{equation*}
    \text{order}(F_i) = \text{order}(\text{amplitude}) + \frac{\text{\# phase variables}}{2} - \frac{\text{\#output + \#input variables}}{4}
\end{equation*}

\begin{equation*}
     = 2 + \frac{1}{2} - \frac{3+3}{4} = 1 \ .
\end{equation*}

So, $F$ is an FIO of order $1$ and is associated to the canonical relation $\Lambda_1^{\prime}\cup \Lambda_2^{\prime}$.

The wavefront relation, $WF^{\prime}(F)$, is the subset of points in $\Lambda_1^{\prime}\cup \Lambda_2^{\prime}$
corresponding to the essential support of the amplitude, $a$, (i.e., in the region outside of which $a$ and all its derivatives decrease faster than any negative power of $\omega$, as $\omega\rightarrow\infty$.) 

\section{Backprojection}\label{surface_section}
To construct an image of the scatterers in the scene we will use backprojection \cite{SARinversion, nolan2004microlocal, nonflat, ambartsoumian2013class, ambartsoumian2018singular, yarman2007bistatic}, which means applying the formal $L^2$-adjoint of the scattering operator to the data. However, as seen in equation \eqref{forward}, the model for the radar data is given as the sum of FIOs. For our analysis, we assume that it is not possible to separate the data into the individual terms corresponding to the signals measured from each emitter before backprojection. As such, an operator cannot be chosen that is simultaneously adjoint to both terms in the sum in \eqref{forward}.
Our approach is to apply the adjoints of $F_1$ and $F_2$ to the data separately,
\begin{align*}
     (F_1^{*}+F_2^{*})FV(\mathbf{z}) =(F_1^{*}+F_2^{*})\left(\left(F_1 + F_2\right)V(\mathbf{z})\right)  \notag \\= F_1^{*}F_1V(\mathbf{z}) + F_1^{*}F_2V(\mathbf{z}) + F_2^{*}F_1V(\mathbf{z}) + F_2^{*}F_2V(\mathbf{z}).
\end{align*}

This leaves us with two types of terms. The first is what we will refer to as diagonal terms. This is when the adjoint operator meets its corresponding part of the forward operator ($F_1^{*}F_1V(\mathbf{z})$ and $F_2^{*}F_2V(\mathbf{z})$), which is desirable in a standard backprojection algorithm. The other terms we refer to as `mixed terms'. The mixed terms are produced when the adjoint operator meets the part of the forward operator that it is not adjoint to ($F_1^{*}F_2V(\mathbf{z})$ and $F_2^{*}F_1V(\mathbf{z})$). This means that we are using \eqref{canonical} to interpret the data but with the incorrect emitter.   Therefore, when we are determining the positions of scatterers that could have resulted in singularities in the data, the mixed terms result in artifacts. We now proceed with an analysis of each type of term.

\subsection{Diagonal operator}\label{diagonal_section}

We denote by $\pi_L$ and $\pi_R$ the projections from the canonical relation $\Lambda_i$ to the two factor spaces $T^{*}\mathbb{R}^2$ on the left and right. The left projection $\pi_L: \Lambda_i\rightarrow T^{*}\mathbb{R}^2$ is defined by
\begin{equation*}
    \pi_L(\mathbf{r},t,\boldsymbol{\rho},\tau;\mathbf{x},\boldsymbol{\xi})\rightarrow(\mathbf{r},t,\boldsymbol{\rho},\tau),
\end{equation*}
and the right projection is defined by
\begin{equation*}
    \pi_R(\mathbf{r},t,\boldsymbol{\rho},\tau;\mathbf{x},\boldsymbol{\xi})\rightarrow(\mathbf{x}, \boldsymbol{\xi}).
\end{equation*}
It is a fact of microlocal analysis that, for a general canonical relation, the singularities
of the projections $\pi_L$ and $\pi_R$ have important implications when studying inverse problems and the relevant linearized forward maps. The singularities \cite{Golubitsky1973} of $\pi_L$ and $\pi_R$ determine whether reconstruction via backprojection is possible, and allows the characterization of artifacts \cite{ambartsoumian2013class, ambartsoumian2018singular, Felea01112005, SARinversion, cusp, EricToddQuinto2013InverseProblemsandImaging}. For any canonical relation, say $C_0 \subset T^{*}\mathbb{R}^n \times T^{*}\mathbb{R}^n
$ associated with an FIO, $F_0$, if
one of the two maps $\pi_L^{\prime}$ and $\pi_R^{\prime}$ is nonsingular at a point $\lambda \in C_0$, then so is the other. As such, we can state a nonsingularity condition for $\pi_L^{\prime}$ and $\pi_R^{\prime}$ in either of the following ways,
\begin{equation}
    det(\pi_L^{\prime})(\lambda) \neq 0 \text{ if and only if } det(\pi_L^{\prime})(\lambda) \neq 0.\label{determinants}
\end{equation}
If we then have that the determinants in \eqref{determinants} is nonzero in a neighbourhood of $\lambda_0$, then $C_0$ is said to be a local canonical graph near
$\lambda_0 = (y_0, \eta_0, x_0, \xi_0)$ in the sense that $C_0$ is the graph of a canonical transformation $\chi:T^{*}\mathbb{R}^n \rightarrow T^{*}\mathbb{R}^n $ defined near $(x_0, \xi_0)$\cite{hormander3}. If $C_0$ is a local canonical graph, the formation of the composition $F^{*}_0F_0$ is covered by the transverse intersection calculus for FIOs \cite{Duistermaat2011, hormander2009analysis}. In this case, $F_0^{*}F_0 \in I^{2m}(\mathcal{D})$ where $\mathcal{D}$ is a canonical relation containing a part of the diagonal relation $\triangle = \{(\mathbf{x}, \boldsymbol{\xi}), (\mathbf{x}, \boldsymbol{\xi})\}$. If $\mathcal{D}$ contains only points of $\triangle$ then $F_0^{*}F_0$ is a pseudodifferential operator \cite{1991Eitt}. Pseudodifferential operators satisfy the pseudo-local property, which tells us that when such operators are applied to a distribution, no singularities will be introduced where the distribution was originally smooth. In imaging contexts, this means that a pseudodifferential operator will not produce artifacts in the image and will yield a faithful reconstruction of the scatterers that are visible in the data. 

If, $\mathcal{D}$ contains points not in $\triangle$ then there will be artifacts in the image. The B\"olker condition, when satisfied, guarantees that this is not the case and subsequently that there will not be artifacts in the image. The B\"olker condition, in a radar imaging context such as the one considered in this article reduces to:
\begin{enumerate}
    \item The projections $\pi_L$ and $\pi_R$ are nonsingular everywhere.
    \item $\pi_L$ is injective.
\end{enumerate}

Operators similar to the diagonal operators here have been studied previously \cite{SARinversion, symes.98} and the associated canonical relations $\Lambda_i$ have been shown to satisfy the B\"olker condition. This means that the diagonal operators $F_i^{*}F_i, i=1,2$, resultant from our backprojection algorithm will not produce artifacts in the image.

\subsection{Mixed operator} \label{mixed}
We now analyse the mixed operator, $F_1^{*}F_2$, which is associated with the contribution to the image as a result of crosstalk between the emitters.

\begin{theorem}\label{composition_theorem}
    \cite{Duistermaat2011} The composition of the FIOs, $F_1$ and $F_2$, is itself an FIO if the following conditions hold:
    \begin{flalign}
        &\bullet \quad \text{Either } \boldsymbol{\xi} \neq 0 \text{ or } \boldsymbol{\zeta} \neq 0 \text{ if } (\mathbf{r}, t, \boldsymbol{\rho}, \tau; \mathbf{x}, \boldsymbol{\xi}) \in \Lambda^{\prime}_1 \text{and } (\mathbf{z}, \boldsymbol{\zeta}; \mathbf{r}, t, \boldsymbol{\rho}, \tau) \in \Lambda^{\prime}_2.\label{composition_condition_1} \\
         &\bullet \quad (\tau, \boldsymbol{\rho}) \neq 0 \text{ if } (\mathbf{r}, t, \boldsymbol{\rho}, \tau; \mathbf{x}, \boldsymbol{\xi}) \in \Lambda^{\prime}_1 \text{ or } (\mathbf{z}, \boldsymbol{\zeta}; \mathbf{r}, t, \boldsymbol{\rho}, \tau) \in \Lambda^{\prime}_2.\label{composition_condition_2}\\
         &\bullet \quad \Lambda_1^{\prime} \times \Lambda_2^{\prime} \text{ intersects } \text{T}^{*}X \times \triangle_{\text{T}^{*}Y}\times\text{T}^{*}Z \text{ transversally. }\label{composition_condition_3}\\
         &\bullet \quad \text{ The projection from } \pi_{X\times Y}(\text{supp }a_1) \times \pi_{Y\times Z}(\text{supp }a_2) \cap X \times \triangle_Y \times Z \text{ into } X\times Z \text{ is a proper mapping. }\label{composition_condition_4}
    \end{flalign}
\end{theorem}

We leave as an exercise to the reader to verify that theorem \ref{composition_theorem} applies to the operators $F_1^{*}$ and $F_2$ and that the composition, $F_1^{*}F_2$ is itself, an FIO.

\begin{theorem}
    The mixed term operator, $F_1^{*}F_2$, is an FIO of order $2$ with canonical relation $\mathcal{C}$ given by
    \begin{equation}
        \mathcal{C} = \left\{( (\mathbf{z}, \boldsymbol{\zeta}), (\mathbf{x},\boldsymbol{\xi})): \mathbf{z} = c(\mathbf{x}-\gamma(\mathbf{r}))+\gamma(\mathbf{r}), \boldsymbol{\zeta} = \boldsymbol{\xi} -\tau\hat{\boldsymbol{\nu}}+\tau\widehat{(\mathbf{x}-\mathbf{E}_2)}\right\},
    \end{equation}
\end{theorem}
where $\boldsymbol{\nu} = c(\mathbf{x}-\gamma(\mathbf{r})+\gamma(\mathbf{r})-\mathbf{E}_1) $.
\begin{proof}
    
$F_1^{*}F_2$ is associated to the composition of canonical relations $\Lambda_1^{\prime \text{ } t } \circ \Lambda_2^{\prime}$, where $ \Lambda_1^{\prime \text{ } t}$ and $\Lambda_2^{\prime}$ are as follows:

\begin{align}
    \Lambda_1^{\prime \text{ } t}= \Big\{((\mathbf{z}, \boldsymbol{\zeta}), (\mathbf{r}, t, \boldsymbol{\rho}, \tau)): t = \frac{1}{c_0}\left(|\mathbf{z}-\gamma(\mathbf{r})| + |\mathbf{z}-\mathbf{E}_1|\right), \notag\\\boldsymbol{\zeta} = -\frac{\tau}{c_0}\left( \widehat{\left(\mathbf{z} - \gamma(\mathbf{r})\right)} + \widehat{(\mathbf{z}-\mathbf{E}_1)} \right), \boldsymbol{\rho} = \frac{\tau}{c_0}\widehat{(\mathbf{z}-\gamma(\mathbf{r}))}_H\Big\},\label{canonical_1}
\end{align}

\begin{align}
    \Lambda_2^{\prime} = \Big\{((\mathbf{r}, t, \boldsymbol{\rho}, \tau), (\mathbf{x}, \boldsymbol{\xi})): t = \frac{1}{c_0}\left(|\mathbf{x}-\gamma(\mathbf{r})| + |\mathbf{x}-\mathbf{E}_2|\right),\notag\\\boldsymbol{\rho} =\frac{\tau}{c_0}\widehat{(\mathbf{x}-\gamma(\mathbf{r}))}_H, \boldsymbol{\xi} = -\frac{\tau}{c_0}\left( \widehat{\left(\mathbf{x} - \gamma(\mathbf{r})\right)} + \widehat{(\mathbf{x}-\mathbf{E}_2)} \right)\Big\}.\label{canonical_2}
\end{align}

Equating $\boldsymbol{\rho}$ in $\Lambda_1^{\prime \text{ } t}$ and $\Lambda_2^{\prime}$ above,
\begin{align}
    \widehat{\left(\mathbf{z}-\gamma(\mathbf{r}) \right)}_H = \widehat{\left(\mathbf{x}-\gamma(\mathbf{r}) \right)}_H,\label{p_equivalence}
\end{align}
\begin{align}
   \implies \left(\mathbf{z}-\gamma(\mathbf{r}) \right)_H = c\left(\mathbf{x}-\gamma(\mathbf{r}) \right)_H,\label{horizontal_components}
\end{align}
Here we note that, because the scatterer is not directly below $\gamma(\mathbf{r})$, it follows that $\left(\mathbf{x}-\gamma(\mathbf{r}) \right)_H \neq 0$ and $c\neq 0$ is a constant that is to be determined. Therefore,
\begin{equation} \label{z_formula}
    \mathbf{z} =  c(\mathbf{x}-\gamma(\mathbf{r})) + \gamma(\mathbf{r}) \ ,
\end{equation}
where we may assume that $c>0$ as the scatterer is assumed to lie at a height which is below the height of the receivers.

We now proceed by equating $t$ in the canonical relations \eqref{canonical_1} and \eqref{canonical_2} to obtain the condition,
\begin{equation}
    |\mathbf{z}-\gamma(\mathbf{r})| + |\mathbf{z}-\mathbf{E}_1|=|\mathbf{x}-\gamma(\mathbf{r})| + |\mathbf{x}-\mathbf{E}_2|,\label{c_cond1}
\end{equation}
which we will use determine the value of the constant, $c$.
Using \eqref{z_formula} here, and remembering that $c>0$, we have
\begin{equation*}
     c|\mathbf{x}-\gamma(\mathbf{r})| + | c(\mathbf{x}-\gamma(\mathbf{r})) + \gamma(\mathbf{r})-\mathbf{E}_1|=|\mathbf{x}-\gamma(\mathbf{r})| + |\mathbf{x}-\mathbf{E}_2|.
\end{equation*}

Grouping similar terms,

\begin{equation*}
    |c(\mathbf{x}-\gamma(\mathbf{r})) + \gamma(\mathbf{r})-\mathbf{E}_1|= (1-c)|\mathbf{x}-\gamma(\mathbf{r})| + |\mathbf{x}-\mathbf{E}_2|,
\end{equation*}
and squaring both sides,
\begin{flalign*}
    \ &\left(c(x_1-r_1) + (r_1-\alpha_1)\right)^2 + \left(c(x_2-r_2) + (r_2-\alpha_2)\right)^2 + \\ &\left(c(x_3-h) + (h-\alpha_3)\right)^2  =\left(|\mathbf{x}-\gamma(\mathbf{r})|+|\mathbf{x}-\mathbf{E}_2|-c|\mathbf{x}-\gamma(\mathbf{r})|\right)^2.
\end{flalign*}
Expanding this results in
\begin{flalign*}
    \ &c^2\big( (x_1-r_1)^2 + (x_2 -r_2)^2 +(x_3-h)^2\big) \\&+ 2c\big((x_1-r_1)(r_1-\alpha_1)+(x_2-r_2)(r_2-\alpha_2) +(x_3-h)(h-\alpha_3)\big)  \\ &+(r_1-\alpha_1)^2+(r_2-\alpha_2)^2+(h-\alpha_3)^2 \\&=  (|\mathbf{x}-\gamma(\mathbf{r})| + |\mathbf{x}-\mathbf{E}_2| )^2 - 2c|\mathbf{x}-\gamma(\mathbf{r})|(|\mathbf{x}-\gamma(\mathbf{r})| + |\mathbf{x}-\mathbf{E}_2| )+c^2|\mathbf{x}-\gamma(\mathbf{r})|^2.
\end{flalign*}
Rearranging and simplifying this expression gives
\begin{flalign*}
    \ &c^2|\mathbf{x}-\gamma(\mathbf{r})|^2 + 2c\big((\mathbf{x}-\gamma(\mathbf{r}))\cdot(\gamma(\mathbf{r})-\mathbf{E}_1)\big) + |\gamma(\mathbf{r}) - \mathbf{E}_1|^2   \\ &=\left(|\mathbf{x}-\gamma(\mathbf{r})| + |\mathbf{x}-\mathbf{E}_2| \right)^2 - 2c|\mathbf{x}-\gamma(\mathbf{r})|\left(|\mathbf{x}-\gamma(\mathbf{r})| + |\mathbf{x}-\mathbf{E}_2| \right)+c^2|\mathbf{x}-\gamma(\mathbf{r})|^2.
\end{flalign*}
Cancellation occurs, leaving us with
\begin{flalign*}
     2c\big((\mathbf{x}-\gamma&(\mathbf{r}))\cdot(\gamma(\mathbf{r})-\mathbf{E}_1)+|\mathbf{x}-\gamma(\mathbf{r})|\left(|\mathbf{x}-\gamma(\mathbf{r})| + |\mathbf{x}-\mathbf{E}_2| \right)\big) \\ &=\left(|\mathbf{x}-\gamma(\mathbf{r})| + |\mathbf{x}-\mathbf{E}_2| \right)^2 -|\gamma(\mathbf{r}) - \mathbf{E}_1|^2.
\end{flalign*}
Therefore,
\begin{align}
     c = \frac{ \left(|\mathbf{x}-\gamma(\mathbf{r})| + |\mathbf{x}-\mathbf{E}_2| \right)^2 -|\gamma(\mathbf{r}) - \mathbf{E}_1|^2}{2\big((\mathbf{x}-\gamma(\mathbf{r}))\cdot(\gamma(\mathbf{r})-\mathbf{E}_1)+|\mathbf{x}-\gamma(\mathbf{r})|\left(|\mathbf{x}-\gamma(\mathbf{r})| + |\mathbf{x}-\mathbf{E}_2| \right)\big) }.\label{c}
\end{align}
To demonstrate that the denominator in (\ref{c}) is always nonzero we consider the terms,

\begin{equation*}
    (\mathbf{x}-\gamma(\mathbf{r}))\cdot(\gamma(\mathbf{r})-\mathbf{E}_1)+|\mathbf{x}-\gamma(\mathbf{r})|\left(|\mathbf{x}-\gamma(\mathbf{r})| + |\mathbf{x}-\mathbf{E}_2| \right),
\end{equation*}
and rewrite this as,
\begin{equation*}
    |\mathbf{x}-\gamma(\mathbf{r})||\gamma(\mathbf{r})-\mathbf{E}_1|\cos\lambda+|\mathbf{x}-\gamma(\mathbf{r})|\left(|\mathbf{x}-\gamma(\mathbf{r})| + |\mathbf{x}-\mathbf{E}_2| \right),
\end{equation*}
where $\lambda$ denotes the angle formed between the vectors $(\mathbf{x}-\gamma(\mathbf{r}))$ and $(\gamma(\mathbf{r})-\mathbf{E}_1)$. The only case needing attention is $\cos\lambda<0$. However,
\begin{equation*}
    \big\vert |\mathbf{x}-\gamma(\mathbf{r})||\gamma(\mathbf{r})-\mathbf{E}_1|\cos\lambda \big\vert \leq |\mathbf{x}-\gamma(\mathbf{r})||\gamma(\mathbf{r})-\mathbf{E}_1|.
\end{equation*}
We can write the right hand side of the above as follows,
\begin{equation*}
    |\mathbf{x}-\gamma(\mathbf{r})||\gamma(\mathbf{r})-\mathbf{E}_1| = |\mathbf{x}-\gamma(\mathbf{r})|| (\gamma(\mathbf{r})-\mathbf{x})+(\mathbf{x}-\mathbf{E}_1)| \ .
\end{equation*}
By the triangle inequality
\begin{equation*}
    |\mathbf{x}-\gamma(\mathbf{r})|| (\gamma(\mathbf{r})-\mathbf{x})+(\mathbf{x}-\mathbf{E}_1)|\leq |\mathbf{x}-\gamma(\mathbf{r})|\left(|\mathbf{x}-\gamma(\mathbf{r})|+|\mathbf{x}-\mathbf{E}_1|\right),
\end{equation*}
where the case of an equality holds only when $\mathbf{x}$ is on the line segment connecting $\gamma(\mathbf{r})$ and $\mathbf{E}_1$ which has been excluded already from consideration. We then conclude that
\begin{equation*}
    \big\vert (\mathbf{x}-\gamma(\mathbf{r}))\cdot(\gamma(\mathbf{r})-\mathbf{E}_1)\big\vert < \big\vert|\mathbf{x}-\gamma(\mathbf{r})|\left(|\mathbf{x}-\gamma(\mathbf{r})| + |\mathbf{x}-\mathbf{E}_2| \right) \big\vert
\end{equation*}
and that the denominator of $c$ never vanishes.

So, the position of an artifact, $\mathbf{z}$, produced by the mixed operator as a result of a scatterer at position $\mathbf{x}$, is given by equation \eqref{z_formula} where $c$ is as in \eqref{c}.

Next, we turn our attention to determining $\boldsymbol{\zeta}$. From \eqref{canonical_1} we have,

\begin{equation*}
    \boldsymbol{\zeta} = -\tau\widehat{(\mathbf{z}-\gamma(\mathbf{r}))} +\tau\widehat{(\mathbf{z}-\mathbf{E}_1)}.
\end{equation*}

We know from \eqref{z_formula} that $\widehat{(\mathbf{z}-\gamma(\mathbf{r}))} = \widehat{(\mathbf{x}-\gamma(\mathbf{r}))}$. Using this and defining the variable $\boldsymbol{\nu} = c(\mathbf{x}-\gamma(\mathbf{r})+\gamma(\mathbf{r})-\mathbf{E}_1)$ gives

\begin{equation*}
     \boldsymbol{\zeta} = -\tau\widehat{(\mathbf{x}-\gamma(\mathbf{r}))} -\tau \hat{\boldsymbol{\nu}} \ .
\end{equation*}

Recalling the value of $\boldsymbol{\xi}$ from \eqref{canonical_2} allows us to write this as

\begin{equation*}
    \boldsymbol{\zeta} = \boldsymbol{\xi} -\tau\hat{\boldsymbol{\nu}} +\tau\widehat{(\mathbf{x}-\mathbf{E}_2)} \ .
\end{equation*}

The order of $F_1^{*}F_2$ is simply the sum of the orders of $F_1^{*}$ and $F_2$, each of which are of order $1$. We then conclude that $F_1^{*}F_2$ is an FIO of order $2$ associated to the artifact relation, $\mathcal{C}$, given by

\begin{equation}
    \mathcal{C} = \left\{( (\mathbf{z}, \boldsymbol{\zeta}), (\mathbf{x},\boldsymbol{\xi})): \mathbf{z} = c(\mathbf{x}-\gamma(\mathbf{r}))+\gamma(\mathbf{r}),  \boldsymbol{\zeta} = \boldsymbol{\xi} -\tau\hat{\boldsymbol{\nu}}+\tau\widehat{(\mathbf{x}-\mathbf{E}_2)}\right\}\label{artifact_relation}\ .
\end{equation}
\end{proof}

\subsubsection{Geometrical description of the propagation of singularities}\label{description}
Figure \ref{fig:fig1} displays an example of how the operator $F_1^{*}F_2$ propagates singularities in the distributions upon which it acts. Here, for the sake of illustration, the receiver, emitters and scatterer are co-planar such that a two-dimensional representation of this phenomenon is sufficient. 

The true scatterer, located at $\mathbf{x}$, (shown as a green mark in figure \ref{fig:fig1}) and its associated artifact, located at $\mathbf{z}$, (shown in purple in figure \ref{fig:fig1}), satisfy equations \eqref{c_cond1} and \eqref{horizontal_components}, placing $\mathbf{x}$ and $\mathbf{z}$ on a pair of ellipsoidal surfaces with foci at $\gamma(\mathbf{r})$ and $\mathbf{E}_2$ and $\gamma(\mathbf{r})$ and $\mathbf{E}_1$ respectively. These ellipsoids share the same sum of distances from points on their surfaces to their foci. 
Equation \eqref{p_equivalence} places both $\mathbf{x}$ and $\mathbf{z}$ on the ray connecting $\gamma(\mathbf{r})$ and $\mathbf{x}$. The artifact relation can then be thought of as a scaling along this ray from its intersection with the first ellipsoidal surface (where the true scatterer is located) to the ray's intersection with the other, temporally equivalent, ellipsoidal surface, which is where the artifact, $\mathbf{z}$, will appear.

\begin{figure}
  \centering
  \includegraphics[width=16cm]{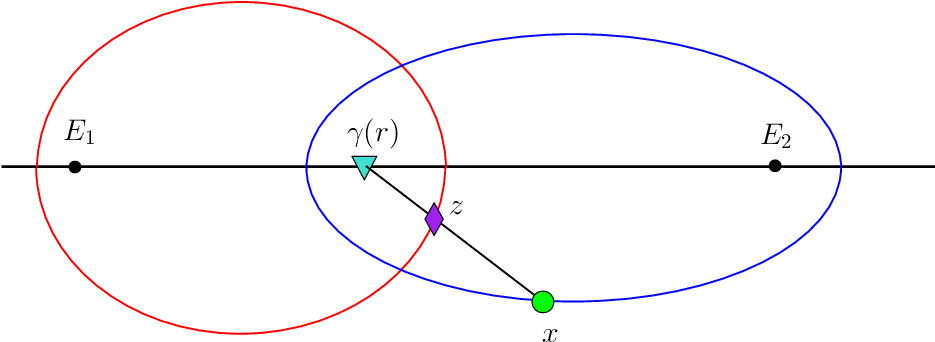}
  \caption{Propagation of singularities as a result of the artifact relation, $\mathcal{C}$, associated with the operator $F_1^{*}F_2$.}
  \label{fig:fig1}
\end{figure}

\section{Mitigation of crosstalk artifacts}\label{mitigation}
In this section we will describe two methods that can be used to avoid the harmful impact of the artifacts that occur as a result of crosstalk between emitters. We first describe how a well selected data acquisition geometry can result in a region of interest that is free from artifacts. Following this, we adapt the method described in \cite{Felea_2007} to generate a series of operators that displace artifacts away from the region we wish to image. We also discuss the requirements on the experimental design for such a technique to be possible. 

\subsection{Data acquisition geometry approach}\label{data_acquisition approach}
Using the insights gained in section \ref{mixed} regarding the locations of artifacts caused by crosstalk, we aim to arrange that the artifacts are outside of a slab shaped ROI between heights $x_3 = 0$ and $x_3 = H$. The minimum height of this slab is arbitrary and can be chosen according to specific applications. For our discussion, we will consider the bottom of the slab to be at ground level, as in many scenarios, this will be desirable when imaging a region above the earth's surface.

The position of the crosstalk artifact, $\mathbf{z}$, as a result of a scatterer located at $\mathbf{x}$, depends also on the location of the receiver. By letting the artifact location, $\mathbf{z}$ be parameterised by $r = (r_1, r_2)$, an "artifact surface" is generated, which is made up of all the individual positions of artifacts associated to receiver locations on the flight track. We examine the region in space that this surface occupies and determine data acquisition geometries such that $\mathbf{z}(\mathbf{r})$ is not contained inside a ROI for all $r$ on the flight track surface. 

We note that this method can come with the potential drawback of needing to omit certain parts of the recorded radar data in order to mitigate crosstalk artifacts. This means that views of the scene from a smaller number of different directions are used, which can result in an image which is lower resolution but is also artifact-free.

\subsubsection{Artifact surface below ground}
Here, we demonstrate how an experimental design can be chosen that places all crosstalk artifacts underground and away from any ROI above ground that is being imaged. The artifact surface is given by
\begin{equation*}
    \mathbf{z}(\mathbf{r}) = c(\mathbf{x}-\gamma(\mathbf{r}))+\gamma(\mathbf{r}).
\end{equation*}

It's vertical component is then
\begin{equation*}
    z_3(\mathbf{r}) = c(x_3-h)+h.
\end{equation*}

 For the artifact surface to be underground we impose $z_3(\mathbf{r}) < 0$, which gives
\begin{equation*}
    c > \frac{h}{h-x_3}.
\end{equation*}

We recall the value of $c$ from \eqref{c} and define the function $\Gamma(\mathbf{x}, r) = c(h-x_3)$. Explicitly, $\Gamma(\mathbf{x}, r)$ is given by

\begin{equation}
     \Gamma(\mathbf{x}, \mathbf{r}) = \frac{ (h-x_3) \left(\left(|\mathbf{x}-\gamma(\mathbf{r})| + |\mathbf{x}-\mathbf{E}_2| \right)^2 -|\gamma(\mathbf{r}) - \mathbf{E}_1|^2\right)}{2\left( |\mathbf{x}-\gamma(\mathbf{r})|\left(|\mathbf{x}-\gamma(\mathbf{r})| + |\mathbf{x}-\mathbf{E}_2| \right)\right)+(\mathbf{x}-\gamma(\mathbf{r}))\cdot(\gamma(\mathbf{r})-\mathbf{E}_1) }.\label{gamma}
\end{equation}

This means that for artifacts to be below the slab we require
\begin{align}
    \Gamma(\mathbf{x},\mathbf{r}) > h.\label{underground_condtion}
\end{align}

If \eqref{underground_condtion} holds for all locations of the receiver, $\gamma(\mathbf{r})$, in the data acquisition geometry, and scatterers, $\mathbf{x}$, in the ROI we will have that the resulting `crosstalk' artifacts will be underground and away from the ROI.

A coarser, but simpler and potentially more practical way to satisfy \eqref{underground_condtion} is described as follows. We denote by $\tilde{\Gamma}(\mathbf{x},r)$ an approximated version of \eqref{gamma}, that still satisfies \eqref{underground_condtion} whenever $\Gamma(\mathbf{x},r)$ does. We construct $\tilde{\Gamma}(\mathbf{x},r)$ by first adding the positive term $|\mathbf{x}-\mathbf{E}_2|^2$ in its denominator, meaning $\Gamma(\mathbf{x},r) > \tilde{\Gamma}(\mathbf{x},r)$. With this we have,

\begin{align*}
   \tilde{\Gamma}(\mathbf{x},\mathbf{r}) =   \frac{ (h-x_3) \left(\left(|\mathbf{x}-\gamma(\mathbf{r})| + |\mathbf{x}-\mathbf{E}_2| \right)^2 -|\gamma(\mathbf{r}) - \mathbf{E}_1|^2\right)}{2\left(|\mathbf{x}-\gamma(\mathbf{r})|^2 + 2|\mathbf{x}-\gamma(\mathbf{r})||\mathbf{x}-\mathbf{E}_2| + |\mathbf{x}-\mathbf{E}_2|^2 + (\mathbf{x}-\gamma(\mathbf{r}))\cdot(\gamma(\mathbf{r})-\mathbf{E}_1)\right) } \ .
\end{align*}

Completing the square in the denominator,

\begin{align*}
   \tilde{\Gamma}(\mathbf{x},\mathbf{r}) = \frac{ (h-x_3) \left(\left(|\mathbf{x}-\gamma(\mathbf{r})| + |\mathbf{x}-\mathbf{E}_2| \right)^2 -|\gamma(\mathbf{r}) - \mathbf{E}_1|^2\right)}{2\left((\left( |\mathbf{x}-\gamma(\mathbf{r})| + |\mathbf{x}-\mathbf{E}_2|\right)^2 + (\mathbf{x}-\gamma(\mathbf{r}))\cdot(\gamma(\mathbf{r})-\mathbf{E}_1)  \right) } \ .
\end{align*}

We recall that the two-way travel time, $c_0 t$, satisfies $c_0 t = |\mathbf{x}-\gamma(\mathbf{r})| + |\mathbf{x}-\mathbf{E}_2| = |\mathbf{x}-\gamma(\mathbf{r})| + |\mathbf{x}-\mathbf{E}_1|$. Using this we express $\tilde{\Gamma}(\mathbf{x},r)$ in terms of $c_o t$,

\begin{equation*}
   \tilde{\Gamma}(\mathbf{x},r) = \frac{(h-x_3) \left( \left(c_0t \right)^2 -|\gamma(\mathbf{r}) - \mathbf{E}_1|^2\right)}{2\left((c_0t)^2+(\mathbf{x}-\gamma(\mathbf{r}))\cdot(\gamma(\mathbf{r})-\mathbf{E}_1)\right) } \ .
\end{equation*}

So, if we can arrange
\begin{equation}
    h < \tilde{\Gamma}(\mathbf{x},\mathbf{r}) < \Gamma(\mathbf{x},\mathbf{r}), \label{under}
\end{equation}

we will have that \eqref{underground_condtion} is satisfied and the artifacts will be underground. This can be done either by backprojecting only the parts of the data for which \eqref{under} holds or instead by arranging the data acquisition geometry such that \eqref{under} holds for every position of the receiver on the flight track.

An example of a setup that yields an artifact surface is displayed in figure \ref{below ground}. Here, a single point scatterer is considered at the point $\mathbf{x}$, and data is measured at receiver locations on the blue surface overhead. The resulting artifact locations, $\mathbf{z}(\mathbf{r})$, are shown as the red surface which is entirely below an example ROI, displayed in green.

\begin{figure}
  \centering
  \includegraphics[width=14cm]{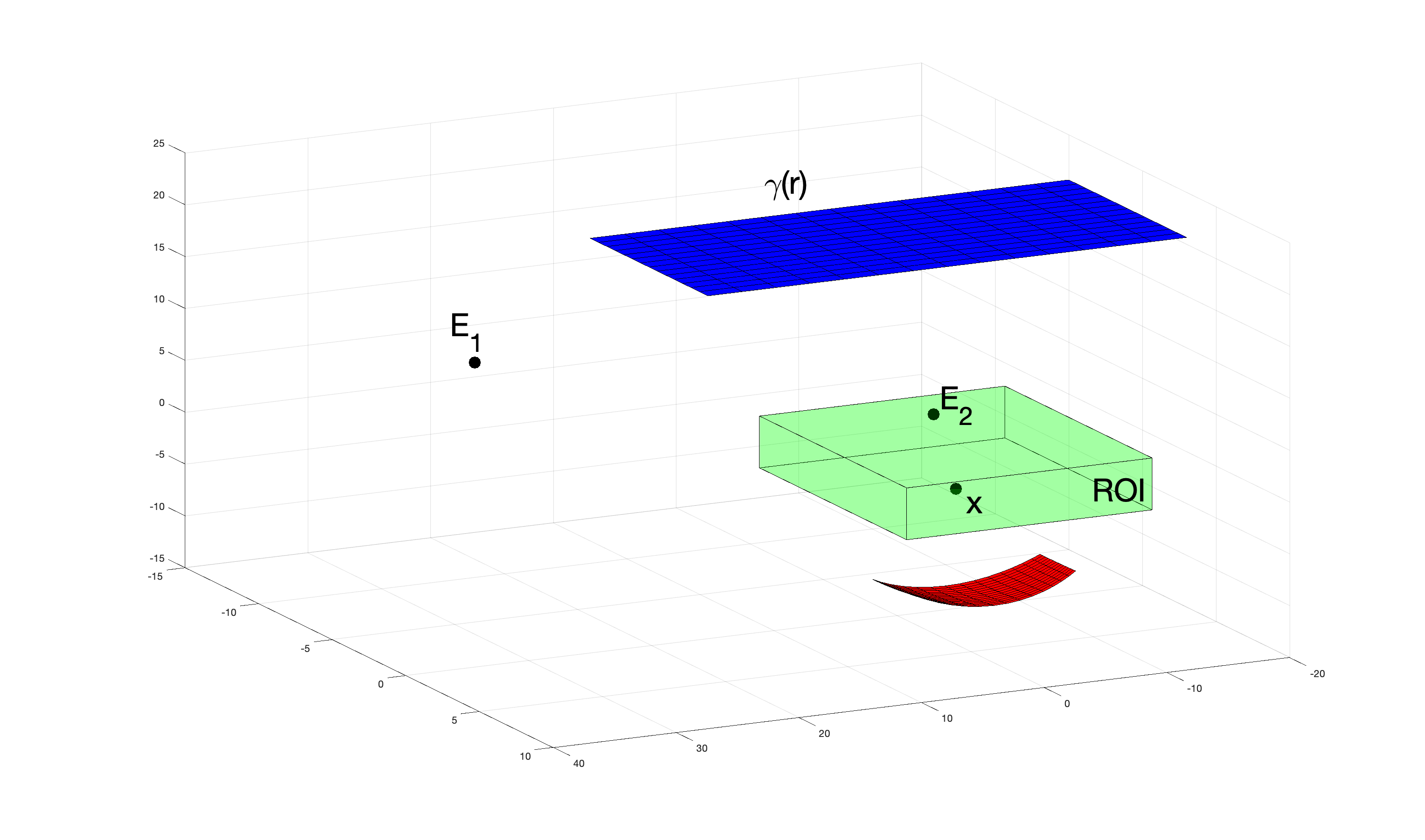}
  \caption{Experimental design involving a single point scatterer at $\mathbf{x}$ that results in artifacts (red) that are underneath a ROI.}
  \label{below ground}
\end{figure}

\subsubsection{Artifact free region above ground}
Next, we determine when the artifacts are above a given height $H$. We set $z_3(\mathbf{r}) > H$, which yields the requirement,

\begin{equation*}
   c < \frac{h-H}{h-x_3}.
\end{equation*}

Using \eqref{c} and \eqref{gamma}, this amounts to

\begin{equation}
    \Gamma(\mathbf{x},\mathbf{r}) < h-H .\label{above_ground_condition}
\end{equation}

Therefore, if \eqref{above_ground_condition} is satisfied we will have that the resulting artifacts must be above height $H$. As before we can develop a potentially more practical way of satisfying \eqref{above_ground_condition}. We denote by $\Bar{\Gamma}(\mathbf{x},r)$ an approximated version of \eqref{gamma} that still satisfies \eqref{above_ground_condition} when $\Gamma(\mathbf{x}, r)$ does. We construct $\Bar{\Gamma}(\mathbf{x},r)$ by first simply omitting the terms in its denominator that are guaranteed to be positive and the terms in its numerator that are negative, meaning $\Gamma(\mathbf{x},r) < \Bar{\Gamma}(\mathbf{x},r)$. We have,

\begin{align*}
   \Bar{\Gamma}(\mathbf{x},\mathbf{r}) = \frac{ (h-x_3) \left(|\mathbf{x}-\gamma(\mathbf{r})| + |\mathbf{x}-\mathbf{E}_2| \right)^2 }{\left( (\mathbf{x}-\gamma(\mathbf{r}))\cdot(\gamma(\mathbf{r})-\mathbf{E}_1)  \right) }
\end{align*}
Again, writing in terms of the two way travel time, $c_0 t$, we are left with

\begin{equation*}
   \Bar{\Gamma}(\mathbf{x},\mathbf{r}) = \frac{(h-x_3)  \left(c_0t \right)^2}{\left((\mathbf{x}-\gamma(\mathbf{r}))\cdot(\gamma(\mathbf{r})-\mathbf{E}_1)\right) }
\end{equation*}

So, if we choose our experimental setup such that
\begin{equation*}
    \Gamma(\mathbf{x},\mathbf{r}) < \Bar{\Gamma}(\mathbf{x},\mathbf{r}) < h-H,
\end{equation*}

we will have that \eqref{above_ground_condition} is automatically satisfied and the artifacts will be located outside and above our region of interest with max height $H$. Again, this can be achieved either by only backprojecting certain parts of the data or by arranging the data acquisition geometry in line with \eqref{above_ground_condition}.

Figure \ref{below ground} shows an example of an experimental design that yields artifacts that are above a ROI, shown in green, due to a point scatterer located at $\mathbf{x}$.

\begin{figure}
  \centering
  \includegraphics[width=16cm]{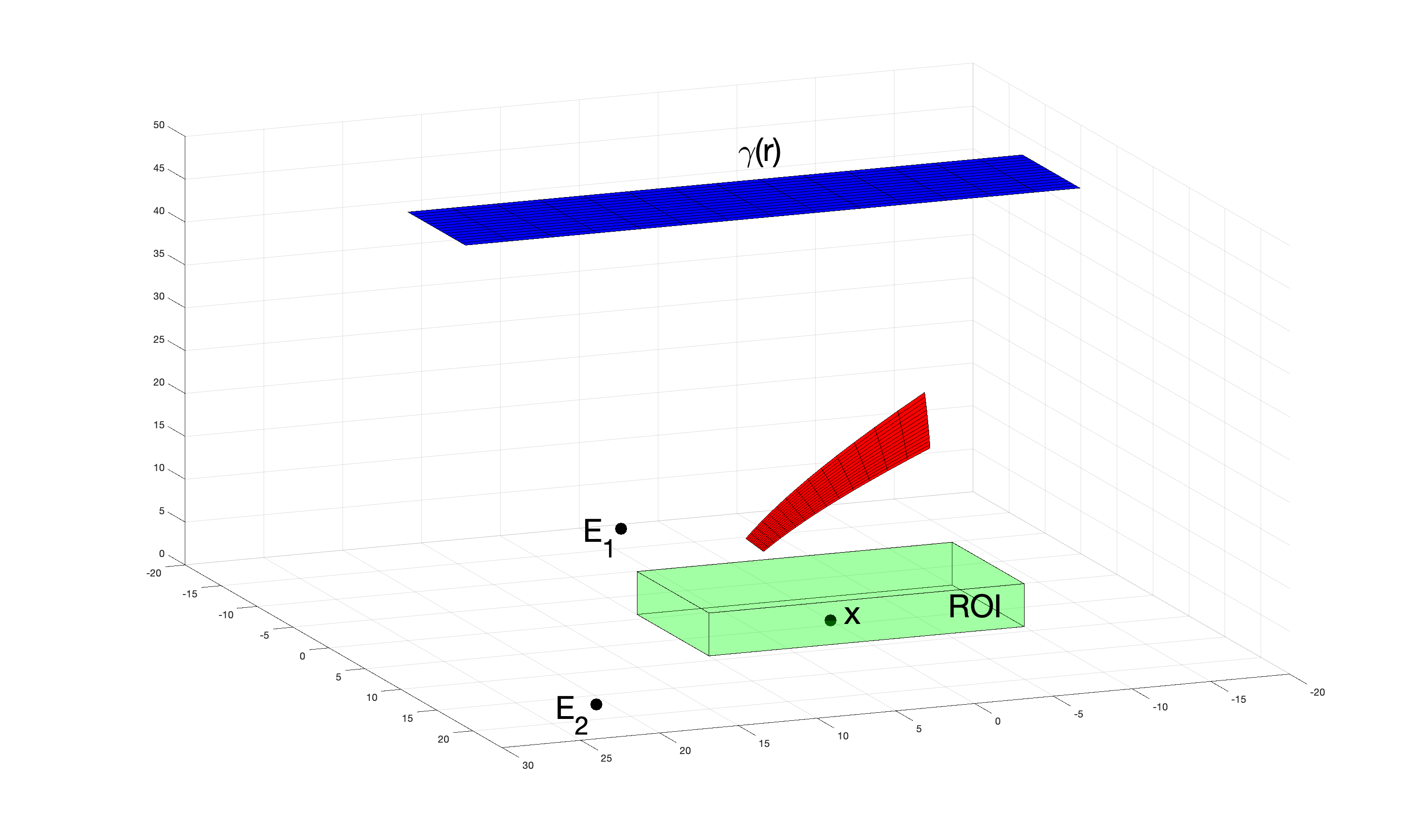}
  \caption{Experimental design with a single point scatterer at $\mathbf{x}$ that results in artifacts (red) that are above a ROI.}
  \label{fig:fig2}
\end{figure}

\subsubsection{Combining conditions}
We also note that \eqref{underground_condtion} and \eqref{above_ground_condition} can be used together to obtain an image where artifacts are outside of the slab. To achieve this, certain data should be excluded from the backprojection. If the converses of \eqref{underground_condtion} and \eqref{above_ground_condition} are satisfied, then the artifacts will be contained within the slab. So, if we omit all data that satisfies
\begin{align}
	h-H<\Gamma(\mathbf{x}, \mathbf{r})< h,\label{subset_condition}
\end{align}

before backprojection, then the image of the ROI will be artifact free as the only resulting artifacts will be above or below the slab. This again comes at the cost of omitting some views of the ROI and, therefore, a possible degradation in the quality of the image.

\subsubsection{A more general region of interest}
The discussions in this section easily can be extended to restrict the artifact surface in any direction, not necessarily just vertically. To illustrate this, we show how one can obtain a radius around the scatterer which is free from artifacts, rather than a slab as we considered previously. To do this we once again start with

\begin{equation*}
    \mathbf{z}(\mathbf{r}) = c(\mathbf{x}-\gamma(\mathbf{r}))+\gamma(\mathbf{r}).
\end{equation*}

We consider the distance from the artifacts, $\mathbf{z}$, to the scatterer, located at $\mathbf{x}$, that produced them, and impose the condition $|\mathbf{z}-\mathbf{x}| > R$ such that a sphere of radius $R$ around the scatterer will be free from artifacts. This gives,
\begin{equation*}
    |\mathbf{z}-\mathbf{x}| = |c(\mathbf{x}-\gamma(\mathbf{r}))+\gamma(\mathbf{r}) - \mathbf{x} | > R
\end{equation*}
\begin{equation*}
     \implies |(c-1)(\mathbf{x}-\gamma(\mathbf{r}))| > R
\end{equation*}
\begin{equation*}
    \implies |c-1||\mathbf{x}-\gamma(\mathbf{r})| > R,
\end{equation*}
\begin{equation}
    \implies |c-1| > \frac{R}{|\mathbf{x}-\gamma(\mathbf{r})|},\label{sphere_condition}
\end{equation}
So to satisfy \eqref{sphere_condition} and ensure that artifacts will be located will be outside of a sphere of radius $R$ centred at the point scatterer located at $\mathbf{x}$ we require,
\begin{equation*}
    c(\mathbf{x},r)< 1- \frac{R}{|\mathbf{x}-\gamma(\mathbf{r})|}, \quad c(\mathbf{x},r)>1+\frac{R}{|\mathbf{x}-\gamma(\mathbf{r})|},
\end{equation*}
where $c$ is as in \eqref{c}. Similarly to before, this can be achieved by omitting data for receiver locations $\gamma(\mathbf{r})$ that satisfy,
\begin{equation*}
     1- \frac{R}{|\mathbf{x}-\gamma(\mathbf{r})|} < c(\mathbf{x}, r) <  1+\frac{R}{|\mathbf{x}-\gamma(\mathbf{r})|}.
\end{equation*}

This could be further generalized by effecting this requirement at all points $\mathbf{x}$ that vary in a more general region of interest.

\subsection{Artifact displacement approach}\label{Artifact displacement approach}

We now move on to our second approach by which one may mitigate the artifacts that are present as a result of crosstalk between emitters. The work here is closely adapted from \cite{Felea_2007}, where the problem considered is that of SAR on a circular flight path. In such a scenario, artifacts appear as reflections of true scatterers across the flight track. A series of FIOs is applied after the imaging operator, that at each iteration, introduce a new artifact into the image while reducing the strength of the pre-existing artifact. At a given iteration, the location of the newest artifact is determined by applying the canonical relation that introduced the artifact in the previous iteration to this same artifact. This is useful because, after each iteration of this process, the newest artifact will be further away from the centre of the flight track than the previous. Meanwhile, the orders of the principal symbol of this composition of FIOs are reduced to $0$ in the locations of all previous artifacts. This means the artifacts of previous iterations will be weakened in the final image. Iterations of this process can be carried out until the artifacts are sufficiently separated from a region of interest. In this section, we show that it is possible to displace artifacts, caused by crosstalk, away from a given region of interest in a similar way.

\subsubsection{When is it possible to displace artifacts?}\label{when are artifacts displaced}
 
In order for it to be possible to displace artifacts in a similar manner to that in \cite{Felea_2007}, we must first check that repeatedly applying the artifact relation, $\mathcal{C}$, given in \eqref{artifact_relation}, to an existing artifact will result in new artifacts that are progressively further away from a ROI.

To illustrate the effect of applying the artifact relation, $\mathcal{C}$, we can extend our discussion in section \ref{description}. We recall that to obtain the artifact location, $\mathbf{z}$, as a result of a scatterer at $\mathbf{x}$, we translate the point $\mathbf{x}$ along the ray connecting $\gamma(\mathbf{r})$ and $\mathbf{x}$ such that $\mathbf{x}$ goes from lying on the ellipsoid with foci at $\gamma(\mathbf{r})$ and $\mathbf{E}_2$ to the ellipsoid with foci at $\gamma(\mathbf{r})$ and $\mathbf{E}_1$. We now wish to apply $\mathcal{C}$ to the artifact at $\mathbf{z}$. So, $\mathbf{z}$ is now translated along the same ray as before. The distance of this translation is such that if we construct an ellipsoid with foci at $\gamma(\mathbf{r})$ and $\mathbf{E}_2$ that contains $\mathbf{z}$, then the translated $\mathbf{z}$, $\tilde{\mathbf{z}}$, will lie on a temporally equivalent ellipsoid with the foci at $\gamma(\mathbf{r})$ and $\mathbf{E}_1$ instead. This process is shown in figure \ref{Displacement of singularities}, where the ellipsoids that are used to determine $\mathbf{z}$ are now dashed and the ellipsoids that determine $\tilde{\mathbf{z}}$ from $\mathbf{z}$ are now heavy. In this way, at each successive application of $\mathcal{C}$, the artifact will traverse either upwards towards the receiver or it will travel downwards and away from the receiver. We will aim to exploit this fact to separate the artifacts from a region of interest. 

\begin{figure}
  \centering
  \includegraphics[width=16cm]{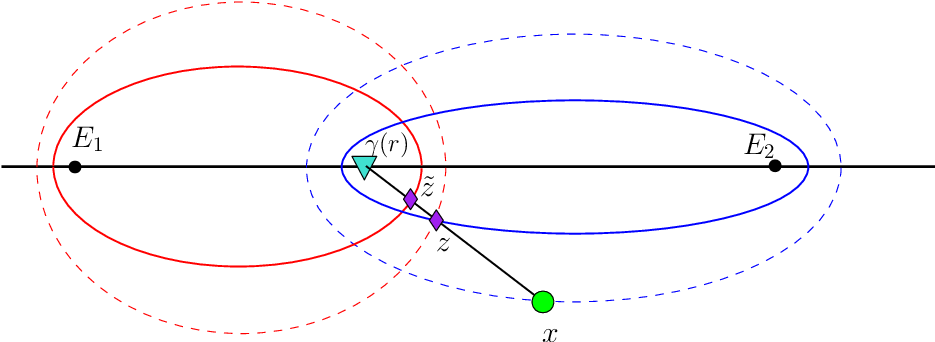}
  \caption{Illustration of the process by which an artifact is displaced from $z$ to $\tilde{z}$ by scaling along the ray connecting $\gamma(r)$ and $x$.}
  \label{Displacement of singularities}
\end{figure}

Next, we wish to rule out certain scenarios that would prevent artifacts from being displaced. For example, it would not make sense to use this method if artifacts looped around within the desired ROI or if the artifact did not move at all. Applying the relation, $\mathcal{C}$, to an artifact will not result in a new artifact that is further away from the ROI if either of the following occur.
\begin{itemize}
    \item The series of displaced artifacts changes direction from travelling away from the receiver to travelling towards it, or vice versa. This could lead to the artifact being trapped in a loop that prevents it from being removed from the region of interest. This could also mean that an artifact that has been displaced out of the ROI could re-enter the ROI at a later iteration. The direction in which the artifact is displaced depends on whether $c<1$ or $c>1$ and this change in direction corresponds to a corresponding change in the magnitude of $c$.
    \item The artifact is not displaced and remains where it is. This happens when the artifact lies on the intersection between the two ellipsoids with foci at $\gamma(\mathbf{r})$ and each $\mathbf{E}_i$. As we will show, this corresponds to $c = 1$ in \eqref{c} and this is where $\mathcal{C} \cap \triangle \neq\emptyset$.
    \item The series of displaced artifacts converge to a point within the ROI. This will occur if $c \rightarrow 1$ as the artifacts approach a point within the ROI.
\end{itemize}

To begin our analysis, we must first determine the locus of points defined by $c=1$. Setting $c = 1$ in \eqref{c},
\begin{flalign*}
     & \left(|\mathbf{x}-\gamma(\mathbf{r})| + |\mathbf{x}-\mathbf{E}_2| \right)^2 -|\gamma(\mathbf{r}) - \mathbf{E}_1|^2   \\&= 2\left((\mathbf{x}-\gamma(\mathbf{r}))\cdot(\gamma(\mathbf{r})-\mathbf{E}_1)+|\mathbf{x}-\gamma(\mathbf{r})|\left(|\mathbf{x}-\gamma(\mathbf{r})| + |\mathbf{x}-\mathbf{E}_2| \right)\right) \ .
\end{flalign*}
Expanding the left hand side,
\begin{flalign*}
     &|\mathbf{x}-\gamma(\mathbf{r})|^2+|\mathbf{x}-\mathbf{E}_2|^2+2|\mathbf{x}-\gamma(\mathbf{r})||\mathbf{x}-\mathbf{E}_2|-|\gamma(\mathbf{r}) - \mathbf{E}_1|^2 \\ &= 2(\mathbf{x}-\gamma(\mathbf{r}))\cdot(\gamma(\mathbf{r})-\mathbf{E}_1) +2|\mathbf{x}-\gamma(\mathbf{r})|^2+2|\mathbf{x}-\gamma(\mathbf{r})||\mathbf{x}-\mathbf{E}_2|.
\end{flalign*}
Some cancellation occurs leaving us with
\begin{equation*}
     |\mathbf{x}-\mathbf{E}_2|^2  = 2(\mathbf{x}-\gamma(\mathbf{r}))\cdot(\gamma(\mathbf{r})-\mathbf{E}_1) +|\mathbf{x}-\gamma(\mathbf{r})|^2+|\gamma(\mathbf{r}) - \mathbf{E}_1|^2.
\end{equation*}
Completing the square yields
\begin{equation*}
     |\mathbf{x}-\mathbf{E}_2|^2  = \left( (\mathbf{x}-\gamma(\mathbf{r})) + (\gamma(\mathbf{r}) - \mathbf{E}_1)\right)^2.
\end{equation*}
Finally we are left with
\begin{equation*}
     |\mathbf{x}-\mathbf{E}_2|^2  = |\mathbf{x}-\mathbf{E}_1|^2,
\end{equation*}
\begin{equation*}
    \implies |\mathbf{x}-\mathbf{E}_2| = |\mathbf{x}-\mathbf{E}_1|.
\end{equation*}

We can similarly obtain that $c>1 \implies |\mathbf{x}-\mathbf{E}_2| > |\mathbf{x}-\mathbf{E}_1|$ and  $c<1 \implies |\mathbf{x}-\mathbf{E}_2| < |\mathbf{x}-\mathbf{E}_1|$. This means that whether the artifact will be displaced towards or away from the receiver depends on what side of the plane, defined by $|\mathbf{x}-\mathbf{E}_1| = |\mathbf{x}-\mathbf{E}_2|$, it is on. Going forward we will denote this plane as $\pi$ . With this, we can now proceed and show that the artifact will be displaced in the same direction throughout all iterations, thus ruling out the possibility that the artifact loops back into the ROI.

First, we consider the case as in figure \ref{Scenario 1}, where $\gamma(\mathbf{r})$ is on the same side of $\pi$ as $\mathbf{E}_2$. Here artifacts on this same side of $\pi$ will be associated to a value of $c<1$. This means that such artifacts on will be displaced towards $\gamma(\mathbf{r})$ and upwards out of the ROI. Artifacts on the other side of $\pi$ will be associated to $c>1$ meaning that they will be pushed away from $\gamma(\mathbf{r})$ and downwards out of the ROI. Points initially lying within $\pi$ are not of concern, as we recall, this is where $c=1$ and $\mathbf{z}=\mathbf{x}$ so we will have true reconstruction here. As all other artifacts are displaced away from $\pi$, no artifact can end up within $\pi$ at a later iteration, where it would otherwise remain for all succeeding iterations. This means that in such a scenario, we can hope to displace all artifacts from the ROI.

The second scenario we must consider is shown in figure \ref{Scenario 2}, where $\gamma(\mathbf{r})$ is now on the same side of $\pi$ as $\mathbf{E}_1$. Points in the ROI on this same side of $\pi$ will be associated with a value of $c>1$, meaning they will be pushed away from the receiver and towards $\pi$. Similarly, on the other side of the plane, where $c<1$, artifacts will be pulled towards $\pi$. We will show that while artifacts in this case are pushed towards $\pi$ they cannot end up within $\pi$ or on the other side of $\pi$.

\begin{proposition}\label{prop}
    Given an artifact at location $\mathbf{z}$, the subsequent location of the artifact, $\tilde{\mathbf{z}}$, which is the result of applying $\mathcal{C}$ to $\mathbf{z}$, will lie on the same side of $\pi$ as $\mathbf{z}$. In other words, $|\mathbf{z}-\mathbf{E}_2| < |\mathbf{z}-\mathbf{E}_1| \implies |\tilde{\mathbf{z}}-\mathbf{E}_2| < |\tilde{\mathbf{z}}-\mathbf{E}_1|$ and $|\mathbf{z}-\mathbf{E}_2| > |\mathbf{z}-\mathbf{E}_1| \implies |\tilde{\mathbf{z}}-\mathbf{E}_2| > |\tilde{\mathbf{z}}-\mathbf{E}_1|$.
\end{proposition}
\begin{proof}
    We consider the case in which $\mathbf{z}$ begins on the side of $\pi$ that is closer to $\mathbf{E}_2$,
    \begin{equation*}
        |\mathbf{z}-\mathbf{E}_2| < |\mathbf{z}-\mathbf{E}_1| \iff c<1.
    \end{equation*}
    We will prove by contradiction that, in this circumstance,
    \begin{equation*}
         |\tilde{\mathbf{z}}-\mathbf{E}_2| < |\tilde{\mathbf{z}}-\mathbf{E}_1|,
    \end{equation*}
    and that $\tilde{\mathbf{z}}$ is on the same side of $\pi$ as $\mathbf{z}$. To this end we assume
    \begin{equation*}
         |\tilde{\mathbf{z}}-\mathbf{E}_1| \leq |\tilde{\mathbf{z}}-\mathbf{E}_2|,
    \end{equation*}
    \begin{equation}
        \implies |\tilde{\mathbf{z}}-\mathbf{E}_1| +|\mathbf{z}-\gamma(\mathbf{r})| \leq |\tilde{\mathbf{z}}-\mathbf{E}_2| +|\mathbf{z}-\gamma(\mathbf{r})|\label{bound}.
    \end{equation}
    The points $\mathbf{z}$ and $\tilde{\mathbf{z}}$ satisfy \eqref{c_cond1},
    \begin{equation*}
        |\mathbf{z}-\mathbf{E}_2|+|\mathbf{z}-\gamma(\mathbf{r})|=|\tilde{\mathbf{z}}-\mathbf{E}_1|+|\tilde{\mathbf{z}}-\gamma(\mathbf{r})|.
    \end{equation*}
    Using \eqref{bound} in the above gives
    \begin{equation*}
        |\mathbf{z}-\mathbf{E}_2|+|\mathbf{z}-\gamma(\mathbf{r})|\leq|\tilde{\mathbf{z}}-\mathbf{E}_2|+|\tilde{\mathbf{z}}-\gamma(\mathbf{r})|.
    \end{equation*}
    We now use the fact that $\tilde{\mathbf{z}} = c(\mathbf{z}-\gamma(\mathbf{r}))+\gamma(\mathbf{r})$ to obtain
    \begin{equation*}
        |\mathbf{z}-\mathbf{E}_2|+|\mathbf{z}-\gamma(\mathbf{r})|\leq|c(\mathbf{z}-\gamma(\mathbf{r}))+\gamma(\mathbf{r})-\mathbf{E}_2|+|c(\mathbf{z}-\gamma(\mathbf{r}))|,
    \end{equation*}
    \begin{equation*}
        \implies |\mathbf{z}-\mathbf{E}_2|+(1-c)|\mathbf{z}-\gamma(\mathbf{r})| \leq |c(\mathbf{z}-\gamma(\mathbf{r}))+\gamma(\mathbf{r})-\mathbf{E}_2|.
    \end{equation*}
    Squaring both sides gives
    \begin{flalign*}
        \left( |\mathbf{z}-\mathbf{E}_2| + |\mathbf{z}-\gamma(\mathbf{r})|\right)^2 - 2c|\mathbf{z}-\gamma(\mathbf{r})|\left( |\mathbf{z}-\mathbf{E}_2| + |\mathbf{z}-\gamma(\mathbf{r})|\right) + c^2|\mathbf{z}-\gamma(\mathbf{r})|^2 \\ \leq c^2|\mathbf{z}-\gamma(\mathbf{r})|^2+2c(\mathbf{z}-\gamma(\mathbf{r}))\cdot(\gamma(\mathbf{r})-\mathbf{E}_2)+|\gamma(\mathbf{r})-\mathbf{E}_2|^2,
    \end{flalign*}

    \begin{equation*}
        \implies c\geq \frac{\left( |\mathbf{z}-\mathbf{E}_2| + |\mathbf{z}-\gamma(\mathbf{r})|\right)^2-|\gamma(\mathbf{r})-\mathbf{E}_2|^2}{2\left((\mathbf{z}-\gamma(\mathbf{r}))\cdot(\gamma(\mathbf{r})-\mathbf{E}_2) + |\mathbf{z}-\gamma(\mathbf{r})|\left(|\mathbf{z}-\mathbf{E}_2|+|\mathbf{z}-\gamma(\mathbf{r})|\right)\right)}.
    \end{equation*}
    Recalling that $c<1$,
    \begin{equation*}
        \frac{\left( |\mathbf{z}-\mathbf{E}_2| + |\mathbf{z}-\gamma(\mathbf{r})|\right)^2-|\gamma(\mathbf{r})-\mathbf{E}_2|^2}{2\left((\mathbf{z}-\gamma(\mathbf{r}))\cdot(\gamma(\mathbf{r})-\mathbf{E}_2) + |\mathbf{z}-\gamma(\mathbf{r})|\left(|\mathbf{z}-\mathbf{E}_2|+|\mathbf{z}-\gamma(\mathbf{r})|\right)\right)} < 1,
    \end{equation*}
    \begin{equation*}
        \implies |\mathbf{z}-\mathbf{E}_2|^2 - |\gamma(\mathbf{r})-\mathbf{E}_2|^2 < 2(\mathbf{z}-\gamma(\mathbf{r}))\cdot(\gamma(\mathbf{r})-\mathbf{E}_2) + |\mathbf{z}-\gamma(\mathbf{r})|^2,
    \end{equation*}
    \begin{equation*}
        \implies |\mathbf{z}-\mathbf{E}_2|^2  < \left( (\mathbf{z}-\gamma(\mathbf{r})) + (\gamma(\mathbf{r}) -\mathbf{E}_2)\right)^2,
    \end{equation*}
    \begin{equation*}
        \implies |\mathbf{z}-\mathbf{E}_2|^2  < |\mathbf{z}-\mathbf{E}_2|^2 .
    \end{equation*}
        \begin{equation*}
        \implies |\mathbf{z}-\mathbf{E}_2|  < |\mathbf{z}-\mathbf{E}_2|.
    \end{equation*}
    This contradiction tells us that the assumption that $|\tilde{\mathbf{z}}-\mathbf{E}_1| \leq |\tilde{\mathbf{z}}-\mathbf{E}_2|$ is false and instead we must have $|\tilde{\mathbf{z}}-\mathbf{E}_2|<|\tilde{\mathbf{z}}-\mathbf{E}_1|.$ 
    
    The proof is similar for the related case where initially $|\mathbf{z}-\mathbf{E}_2| > |\mathbf{z}-\mathbf{E}_1|$ where the subsequent artifact location, $\tilde{\mathbf{z}}$, satisfies $|\tilde{\mathbf{z}}-\mathbf{E}_2|>|\tilde{\mathbf{z}}-\mathbf{E}_1|$.
\end{proof}

This result makes sense geometrically, as $\mathbf{z}$ and $\tilde{\mathbf{z}}$ are located at the intersections of a ray emanating from $\gamma(\mathbf{r})$ with two ellipsoidal surfaces that have a common focus at $\gamma(\mathbf{r})$. These ellipsoidal surfaces intersect within the plane $\pi$. Then, as the ray emanates from within the pair of ellipsoids it must pierce both on the same side of $\pi$, meaning $\mathbf{z}$ and $\tilde{\mathbf{z}}$ will lie on the same side of $\pi$. 

We now know that throughout all iterations artifacts will be propagated in a single direction. They will either be pushed or pulled from $\gamma(\mathbf{r})$ indefinitely, or will iteratively approach $\pi$. The first two possibilities mean that artifacts will raise above or lower to be below a ROI. Then with the restriction that $\pi$ does not intersect the ROI we guarantee that the artifacts that approach $\pi$ will eventually leave the ROI. This restriction can be satisfied by appropriately designing the experiment where possible. In cases where it cannot be arranged that $\pi$ is separated from the ROI then, as we will discuss in the next subsection, beam forming can be used to achieve a situation where artifacts will still be displaced from the ROI.

Finally, we note that the distance that an artifact is displaced is governed by the magnitude of $c$. The magnitude of $c$ is greater at points further away from $\pi$. With this in mind, it makes sense to, if possible, arrange that the ROI is as far away from $\pi$ as possible. This means that artifacts will be displaced further at each iteration and fewer iterations will be required to remove all artifacts from the ROI. 

\begin{figure}
  \centering
  \includegraphics[width=10cm]{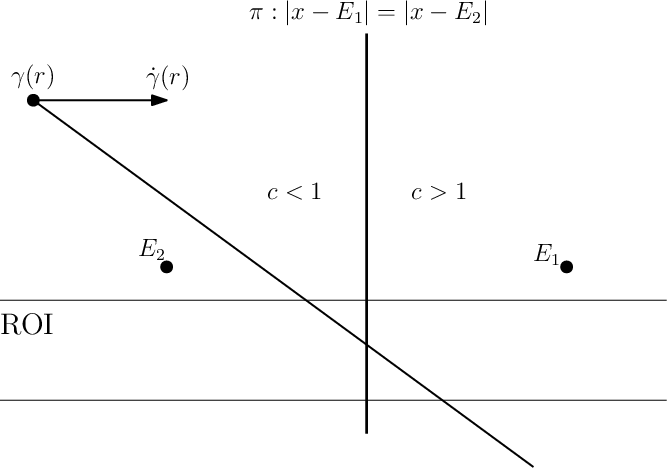}
  \caption{An experimental design where artifacts on the same side of $\pi$ as $\gamma(r)$ will be associated to a value of $c<1$.}
  \label{Scenario 1}
\end{figure}

\begin{figure}
  \centering
  \includegraphics[width=10cm]{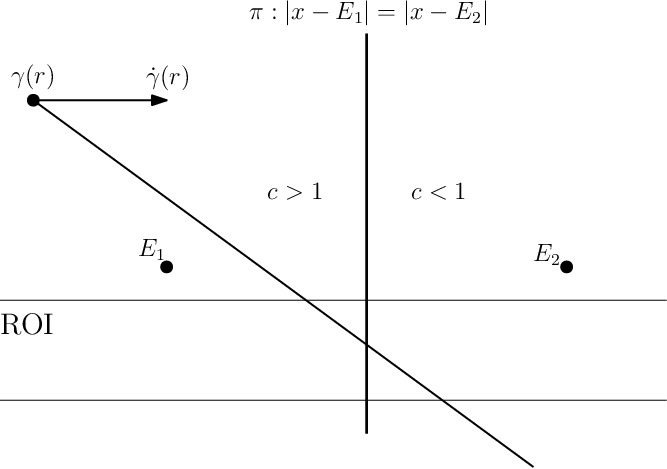}
  \caption{An experimental design where artifacts on the same side of $\pi$ as $\gamma(r)$ will be associated to a value of $c>1$.}
  \label{Scenario 2}
\end{figure}

\subsubsection{Role of beam forming}\label{beam_forming}

Here, we outline how beam forming can be used to achieve a scenario in which crosstalk artifacts can always be displaced away from the desired ROI. As discussed in section \ref{when are artifacts displaced}, for the artifacts to be displaced we must eliminate the possibility that rays, connecting points in the ROI to receiver locations on the flight track, intersect $\pi$ within the ROI. Beam forming can be utilised to remove all such rays that intersect $\pi$ in the ROI, making it possible to displace artifacts when it would otherwise not be. 

With this in mind, we consider the setup displayed in figure \ref{fig:fig4}, where the ROI is slab shaped, between ground level and a fixed height above ground. For a given receiver location, $\gamma(\mathbf{r})$, we wish to determine the critical angle $\theta_c$, that is formed between the ray emanating from $\gamma(\mathbf{r})$ and the horizontal flight track, such that the ray intersects $\pi$ at ground level and at the base of the ROI. Any beam that forms an angle of or less than $\theta_c$ should be removed via beam forming as they will intersect $\pi$ within the ROI and thus prevent us from being able to displace artifacts. All other beams that form an angle greater than $\theta_c$ will be of no concern as these rays will intersect $\pi$ below ground level and outside the ROI.

We let $\mathbf{x}_b$ denote the point on the ground where the ray emanating from $\gamma(\mathbf{r})$ and the plane $\pi$ intersect. Letting $\mathbf{E}_1 = (\alpha_1, \alpha_2, \alpha_3)$ and $\mathbf{E}_2 = (\Tilde{\alpha}_1, \Tilde{\alpha}_2, \Tilde{\alpha}_3)$, we first find the line in the horizontal plane where $\pi$ intersects the ground by setting $x_3 = 0$ in $|\mathbf{x}-\mathbf{E}_1| = |\mathbf{x}-\mathbf{E}_2|$, giving

\begin{equation*}
     \sqrt{(x_1-\alpha_1)^2 + (x_2 - \alpha_2)^2 + \alpha_3^2} = \sqrt{(x_1-\Tilde{\alpha}_1)^2 + (x_2 - \Tilde{\alpha}_2)^2 + \Tilde{\alpha}_3^2}.
\end{equation*}

 Rearranging this gives the equation of the line in the $x_1x_2$ plane corresponding to the intersection of $\pi$ with the ground,
 
 \begin{equation}
     x_2 = x_1 \left(\frac{\Tilde{\alpha}_1 - \alpha_1}{\alpha_2 - \Tilde{\alpha}_2}\right)+\frac{(\alpha_1^2 + \alpha_2^2 + \alpha_3^2) - (\Tilde{\alpha}_1^2 + \Tilde{\alpha}_2^2 + \Tilde{\alpha}_3^2)}{2(\alpha_2 - \Tilde{\alpha}_2)}\label{line1}.
 \end{equation}
 
The slope of this line is

\begin{equation*}
    m_1 = \frac{\Tilde{\alpha}_1 - \alpha_1}{\alpha_2 - \Tilde{\alpha}_2}.
\end{equation*}

Next, we are interested in the ray emanating from $\gamma(\mathbf{r})$ with horizontal components perpendicular to the intersection of $\pi$ with the ground. In particular, the projection of this ray onto the ground will have slope perpendicular to $m_1$, given by
 \begin{equation*}
     m_2 = \frac{\Tilde{\alpha}_2 - \alpha_2}{\Tilde{\alpha}_1 - \alpha_1 }.
 \end{equation*}
 The horizontal projection of this ray will also pass through the point $(r_1, r_2)$ and therefore has equation

 \begin{equation}
     x_2 = \left(\frac{\Tilde{\alpha}_2 - \alpha_2}{\Tilde{\alpha}_1 - \alpha_1 }\right)(x_1-r_1)+r_2. \label{line2}
 \end{equation}

The point $\mathbf{x}_b$ is the intersection of \eqref{line1} and \eqref{line2} and has coordinates $\mathbf{x}_b = (\Bar{x_1}, \Bar{x_2}, 0)$ where
\begin{flalign*}
    \Bar{x_1} = &\left(\frac{(\Tilde{\alpha}_1-\alpha_1)(\alpha_2 - \Tilde{\alpha}_2)}{(\alpha_1 - \Tilde{\alpha}_1)(\Tilde{\alpha}_1-\alpha_1)+(\Tilde{\alpha}_2-\alpha_2)(\alpha_2 - \Tilde{\alpha}_2)}\right)\\& \left( \frac{(\alpha_1^2 + \alpha_2^2 + \alpha_3^2) - (\Tilde{\alpha}_1^2 + \Tilde{\alpha}_2^2 + \Tilde{\alpha}_3^2)}{2(\alpha_2 - \Tilde{\alpha}_2)} +r_1\left(\frac{\Tilde{\alpha}_2-\alpha_2}{\Tilde{\alpha}_1-\alpha_1}\right)-r_2\right),
\end{flalign*}
 and 
 \begin{equation*}
     \Bar{x_2} = \left(\frac{\Tilde{\alpha}_2 - \alpha_2}{\Tilde{\alpha}_1 - \alpha_1 }\right)(\Bar{x_1}-r)+r_2.
 \end{equation*}

 We then let $\mathbf{x}_h$ be the point at the same height as $\gamma(\mathbf{r})$ whose horizontal projection is $\mathbf{x}_b$, i.e., $\mathbf{x}_h= (\Bar{x_1},\Bar{x_2},h)$. With knowledge of these points, we can calculate the value of the angle $\theta_c$ such that the ray emanating from $\gamma(\mathbf{r})$ will intersect $\pi$ on the ground. Any beam forming an angle greater than $\theta_c$ with the flight track will intersect $\pi$ underground as desired. Using our knowledge of $x_h$ and basic trigonometry we have that

\begin{equation*}
    \theta_c(\mathbf{r}) = \arctan\left(\frac{h}{|\mathbf{x}_h(\mathbf{r}) - \gamma(\mathbf{r})|}\right).
\end{equation*}

We specify the dependence of $\theta_c$ on $r$ as this angle will vary with the location of the receiver.

\begin{figure}
  \centering
  \includegraphics[width=10cm]{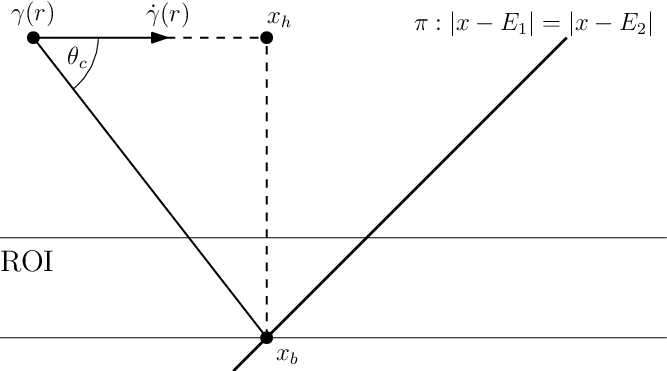}
  \caption{Beam forming can be utilised to omit beams that form an angle less than $\theta_c$ with the horizontal flight track, thus preventing the intersection of rays with $\pi$ within the ROI and facilitating the displacement of artifacts.}
  \label{fig:fig4}
\end{figure}

\subsubsection{Combining backprojection operators}\label{combining} 
We note here that our discussions thus far have been in the context of the mixed operator $F_1^{*}F_2$. Had we chosen $F_2^{*}$ as our backprojection operator rather than $F_1^{*}$ we would have obtained the mixed term operator $F_2^{*}F_1$. Equivalent analysis of this operator can be carried out to show that in this case that $c<1 \implies |\mathbf{x}-\mathbf{E}_2| > |\mathbf{x}-\mathbf{E}_1|$ and  $c>1 \implies |\mathbf{x}-\mathbf{E}_2| < |\mathbf{x}-\mathbf{E}_1|$, which is the reverse as that obtained when considering the operator $F_1^{*}F_2$. In this way, changing backprojection operator from $F_1^{*}$ to $F_2^{*}$ will correspond to a change in the value of $c$ that changes the direction in which the artifacts are displaced.

This can be used to achieve a scenario where $F_1^{*}$ is applied to the data recorded on one side of $\pi$ and $F_2^{*}$ is applied to the data recorded on the other side. This would be done such that the artifacts can be displaced out of the ROI on both sides, even when $\pi$ intersects the ROI. Then, the images on either side of the plane $|\mathbf{x}-\mathbf{E}_1| = |\mathbf{x}-\mathbf{E}_2|$ can be combined to give an image of the entire ROI. This means that we can choose an appropriate backprojection operator for either side of $\pi$ and then construct a series of operators that filter the artifacts from each side of the plane respectively. Then putting the two sections together will result in an image that is completely free from artifacts. 

Finally, we mention the fact that when beam forming is used to exclude rays passing from one side of $|\mathbf{x}-\mathbf{E}_1| = |\mathbf{x}-\mathbf{E}_2|$ to the other within the ROI there will be less coverage of the region near to this plane. This deficiency can be combated with the inclusion of another emitter, $\mathbf{E}_3$, that is positioned such that the plane $|\mathbf{x}-\mathbf{E}_1| = |\mathbf{x}-\mathbf{E}_2|$ is separated from $|\mathbf{x}-\mathbf{E}_1| = |\mathbf{x}-\mathbf{E}_3|$. We can use the data from $\mathbf{E}_3$ to provide full coverage of the region near $|\mathbf{x}-\mathbf{E}_1| = |\mathbf{x}-\mathbf{E}_2|$ that was previously omitted. Similarly, there will be issues of coverage near the plane $|\mathbf{x}-\mathbf{E}_1| = |\mathbf{x}-\mathbf{E}_3|$. However, we can use $\mathbf{E}_2$ to provide full coverage of this region. In this way we can use combination of the three backprojection operators $F_1^{*}$, $F_2^{*}$ and $F_3^{*}$ to create three sections of the final image, in all of which, the artifacts can be displaced out of the ROI. Then when artifacts have been displaced and the three sections are combined, we will have a full coverage and artifact free image of the ROI.

\subsection{Calculating series of FIOs to filter artifacts}\label{filtering_process}

In this section, we adapt the calculation from \cite{Felea_2007} to determine a series of FIOs which will displace the artifacts as discussed previously. We first state the following theorem that relates the principal symbol of FIOs to that of their composition.

\begin{theorem}
    \cite{Duistermaat2011} If $A_1 \in I^{m_1}(X, Y, C_1)$ and $A_2 \in I^{m_2}(Y, Z, C_2)$, with principal symbols $\sigma_{A_1}$, $\sigma_{A_2}$ respectively, and if $C_1 \times C_2$ intersects $T^{*}X \times \triangle_{T^{*}Y} \times T^{*}Z$ transversally, then $A_1A_2 \in I^{m_1+m_2} (X, Z, C_1 \circ C_2) $ and its principal symbol, $\sigma_{A_1 A_2}$, is given by
\begin{equation*}
    \sigma_{A_1 A_2} (\mathbf{x}, \boldsymbol{\xi}; \mathbf{z}, \boldsymbol{\zeta}) = \sum_{S} \sigma_{A_1} (\mathbf{x}, \boldsymbol{\xi} , \mathbf{y}, \boldsymbol{\eta})\sigma_{A_2} (\mathbf{y}, \boldsymbol{\eta}, \mathbf{z}, \boldsymbol{\zeta} ),
\end{equation*}
where $S = \{(\mathbf{y}, \boldsymbol{\eta}) : (\mathbf{x}, \boldsymbol{\xi} , \mathbf{y}, \boldsymbol{\eta}) \in C_1 \text{ and } (\mathbf{y}, \boldsymbol{\eta}, \mathbf{z}, \boldsymbol{\zeta} ) \in C_2\}$.
\end{theorem}

\begin{definition}
        \cite{Duistermaat2011} The operator $P$ is called microlocally elliptic if it's principal symbol $\sigma(\mathbf{x},\boldsymbol{\xi})$ satisfies
    \begin{equation*}
        \xi \neq 0 \implies \sigma(\mathbf{x},\boldsymbol{\xi})\neq 0 
    \end{equation*}
\end{definition}

Let $F=F_1 + F_2$. Provided the experiment is designed in line with the discussions in section \ref{when are artifacts displaced} such that $\triangle$ does not intersect $\text{Gr}(\mathcal{C})$ we will have that $F_1^{*}F\in I^{2m}(\triangle)+I^{2m}(\text{Gr}(\mathcal{C}))$. We will construct the first operator in the series, $Q_1$, such that $Q_1 \in I^{0}(\triangle)+I^{0}(\text{Gr}(\mathcal{C}))$. It follows that
\begin{equation*}
    Q_1F_1^{*}F\in I^{2m}(\triangle)+I^{2m}(\text{Gr}(\mathcal{C}))+I^{2m}(\text{Gr}(\mathcal{C}^2))
\end{equation*}
There is one contribution to $\triangle$ from $\triangle\circ \triangle$ and there are two contributions to $\text{Gr}(\mathcal{C})$, from $\triangle\circ \mathcal{C}$ and $\mathcal{C}\circ\triangle$.

We have on $\triangle$,
\begin{equation*}
      \sigma_{Q_1F_1^{*}F}|_\triangle(\mathbf{x}, \boldsymbol{\xi}; \mathbf{x}, \boldsymbol{\xi}) = \sigma_{Q_1F_1^{*}F}|_{\triangle\circ\triangle}(\mathbf{x}, \boldsymbol{\xi}; \mathbf{x}, \boldsymbol{\xi})=\sigma_{Q_1}|_{\triangle}(\mathbf{x}, \boldsymbol{\xi}; \mathbf{x}, \boldsymbol{\xi})\sigma_{F_1^{*}F}|_{\triangle}(\mathbf{x}, \boldsymbol{\xi}; \mathbf{x}, \boldsymbol{\xi}) .
\end{equation*}
On $\text{Gr}(\mathcal{C})$ we have
\begin{equation*}
    \sigma_{Q_1F_1^{*}F}|_{\text{Gr}(\mathcal{C})} (\mathcal{C}(\mathbf{x}, \boldsymbol{\xi});\mathbf{x}, \boldsymbol{\xi})=\sigma_{Q_1F_1^{*}F}|_{\triangle\circ\text{Gr}(\mathcal{C})} (\mathcal{C}(\mathbf{x}, \boldsymbol{\xi});\mathbf{x}, \boldsymbol{\xi}) + \sigma_{Q_1F_1^{*}F}|_{\text{Gr}(\mathcal{C})\circ\triangle}(\mathcal{C}(\mathbf{x}, \boldsymbol{\xi});\mathbf{x}, \boldsymbol{\xi})
\end{equation*}
\begin{equation*}
    =\sigma_{Q_1}|_{\triangle}(\mathbf{x}, \boldsymbol{\xi}; \mathbf{x}, \boldsymbol{\xi})\sigma_{F_1^{*}F}|_{\text{Gr}(\mathcal{C})}(\mathcal{C}(\mathbf{x}, \boldsymbol{\xi});\mathbf{x}, \boldsymbol{\xi})+\sigma_{Q_1}|_{\text{Gr}(\mathcal{C})}(\mathcal{C}(\mathbf{x}, \boldsymbol{\xi});\mathbf{x}, \boldsymbol{\xi})\sigma_{F_1^{*}F}|_{\triangle}(\mathbf{x}, \boldsymbol{\xi}; \mathbf{x}, \boldsymbol{\xi}).
\end{equation*}
Proceeding, we will drop the dependence on the variables $(\mathbf{x},\boldsymbol{\xi})$ from each of the principal symbols. We will construct $Q_1$ such that $Q_1F_1^{*}F$ is elliptic on $\triangle$ and $\text{Gr}(\mathcal{C}^2)$ and it's amplitude vanishes on $\text{Gr}(\mathcal{C})$. This means that the original artifact will be weakened and the true scatterer as well as displaced artifact will be left in the image at the same strength. We note also that this calculation operates at principal symbol level when there are in fact lower order contributions from these FIOs. However, the validity of this approximation will be borne out in the numerical simulations later.

Therefore, we wish to choose $Q_1$ so that the following is valid:
\begin{equation*}
    \sigma_{Q_1F_1^{*}F}|_\triangle =\omega^{2m},
\end{equation*}
\begin{equation*}
    \sigma_{Q_1F_1^{*}F}|_{\text{Gr}(\mathcal{C})} = 0.
\end{equation*}
Therefore, we want to solve the above system for $\sigma_{Q_1}|_{\triangle}$ and $\sigma_{Q_1}|_{\text{Gr}(\mathcal{C})}$. The system is
\begin{equation*}
    \begin{pmatrix}
        \sigma_{F_1^{*}F}|_{\triangle} & 0 \\
        \sigma_{F_1^{*}F}|_{\text{Gr}(\mathcal{C})} & \sigma_{F_1^{*}F}|_{\triangle}
    \end{pmatrix}
    \begin{pmatrix}
        \sigma_{Q_1}|_{\triangle}\\
        \sigma_{Q_1}|_{\text{Gr}(\mathcal{C})}
    \end{pmatrix}
    =\begin{pmatrix}
        \omega^{2m} \\
        0
    \end{pmatrix}.
\end{equation*}
We have that
\begin{equation*}
    \det\begin{pmatrix}
        \sigma_{F_1^{*}F}|_{\triangle} & 0 \\
        \sigma_{F_1^{*}F}|_{\text{Gr}(\mathcal{C})} & \sigma_{F_1^{*}F}|_{\triangle}
    \end{pmatrix} = (\sigma_{F_1^{*}F}|_{\triangle})^2 \neq 0,
\end{equation*}
(given that $F_1^{*}F$ is elliptic on $\triangle$).
By Cramer's rule:
\begin{equation}
    \sigma_{Q_1}|_{\triangle} = \frac{\omega^{2m}\sigma_{F_1^{*}F}|_{\triangle}}{(\sigma_{F_1^{*}F}|_{\triangle})^2}= \frac{\omega^{2m}}{\sigma_{F_1^{*}F}|_{\triangle}},\label{amp_diag}
\end{equation}
\begin{equation}
    \sigma_{Q_1}|_{\text{Gr}(\mathcal{C})}=\frac{-\omega^{2m}\sigma_{F_1^{*}F}|_{\text{Gr}(\mathcal{C})}}{(\sigma_{F_1^{*}F}|_{\triangle})^2},\label{amp_mixed}
\end{equation}
giving us the required amplitudes to construct the operator $Q_1$.

A similar calculation is used to construct the operator at the second iteration, $Q_2$. Before proceeding we note that our discussions in section ensure that there is no intersection between $\text{Gr}(\mathcal{C}^i)$ and $\text{Gr}(\mathcal{C}^j)$ when $i\neq j$. We then have that $Q_1F_1^{*}F\in I^{2m}(\triangle)+I^{2m}(\text{Gr}(\mathcal{C}))+I^{2m}(\text{Gr}(\mathcal{C}^2))$ and we choose $Q_2\in I^{0}(\triangle)+I^{0}(\text{Gr}(\mathcal{C}))+I^{0}(\text{Gr}(\mathcal{C}^2))$. Then $Q_2Q_1F_1^{*}F\in I^{2m}(\triangle)+I^{2m}(\text{Gr}(\mathcal{C}))+I^{2m}(\text{Gr}(\mathcal{C}^2))+I^{2m}(\text{Gr}(\mathcal{C}^3))+I^{2m}(\text{Gr}(\mathcal{C}^4))$. We will choose $Q_2$ such that $Q_2Q_1F_1^{*}F$ is elliptic on $\triangle\cup\text{Gr}(\mathcal{C}^4)$ and has order $2m-1$ on $\text{Gr}(\mathcal{C}) \cup \text{Gr}(\mathcal{C}^2)\cup \text{Gr}(\mathcal{C}^3)$ by  solving the following system,
\begin{equation}
    \sigma_{Q_2Q_1F_1^{*}F}|_\triangle =\omega^{2m}\label{req1}
\end{equation}
\begin{equation}
    \sigma_{Q_2Q_1F_1^{*}F}|_{\text{Gr}(\mathcal{C})} = 0\label{req2}
\end{equation}
\begin{equation}
    \sigma_{Q_2Q_1F_1^{*}F}|_{\text{Gr}(\mathcal{C}^2)} = 0\label{req3}
\end{equation}
\begin{equation}
    \sigma_{Q_2Q_1F_1^{*}F}|_{\text{Gr}(\mathcal{C}^3)} = 0\label{req4}
\end{equation}

There is one contribution to $\triangle$ from $\triangle\circ \triangle$, so \eqref{req1} gives,
 \begin{equation*}
      \sigma_{Q_2Q_1F_1^{*}F}|_\triangle = \sigma_{Q_2Q_1F_1^{*}F}|_{\triangle\circ\triangle}=\sigma_{Q_2}|_{\triangle}\sigma_{Q_1F_1^{*}F}|_{\triangle} = \omega^{2m}
\end{equation*}

There are two contributions to $\text{Gr}(\mathcal{C})$: $\triangle\circ\mathcal{C}$ and $\mathcal{C}\circ\triangle$, so \eqref{req2} gives,
 \begin{equation*}
      \sigma_{Q_2Q_1F^{*}F}|_{\text{Gr}(\mathcal{C})} = \sigma_{Q_2Q_1F^{*}F}|_{\text{Gr}(\mathcal{C})\circ\triangle}+\sigma_{Q_2Q_1F^{*}F}|_{\triangle\circ\text{Gr}(\mathcal{C})}=\sigma_{Q_2}|_{\text{Gr}(\mathcal{C})}\sigma_{Q_1F_1^{*}F}|_{\triangle}+\sigma_{Q_2}|_{\triangle}\sigma_{Q_1F_1^{*}F}|_{\text{Gr}(\mathcal{C})}
\end{equation*}
Using the facts that $\sigma_{Q_1F_1^{*}F}|_{\text{Gr}(\mathcal{C})} = 0$ and $\sigma_{Q_1F_1^{*}F}|_{\triangle} \neq 0$ gives,
\begin{equation*}
    \sigma_{Q_2}|_{\text{Gr}(\mathcal{C})} = 0
\end{equation*}

There are three contributions to $\text{Gr}(\mathcal{C}^2)$: $\triangle\circ\mathcal{C}^2$ and $\mathcal{C}^2\circ\triangle$ and $\mathcal{C}\circ\mathcal{C}$, so \eqref{req3} gives,
\begin{equation*}
      \sigma_{Q_2Q_1F^{*}F}|_{\text{Gr}(\mathcal{C}^2)} = \sigma_{Q_2Q_1F^{*}F}|_{\text{Gr}(\mathcal{C}^2)\circ\triangle}+\sigma_{Q_2Q_1F^{*}F}|_{\triangle\circ\text{Gr}(\mathcal{C}^2)} + \sigma_{Q_2Q_1F^{*}F}|_{\text{Gr}(\mathcal{C})\circ\text{Gr}(\mathcal{C})}
\end{equation*}
\begin{equation*}
     =\sigma_{Q_2}|_{\text{Gr}(\mathcal{C}^2)}\sigma_{Q_1F_1^{*}F}|_{\triangle}+\sigma_{Q_2}|_{\triangle}\sigma_{Q_1F_1^{*}F}|_{\text{Gr}(\mathcal{C}^2)} + \sigma_{Q_2}|_{\text{Gr}(\mathcal{C})} \sigma_{Q_1F_1^{*}F}|_{\text{Gr}(\mathcal{C})}
\end{equation*}
Using $\sigma_{Q_1F_1^{*}F}|_{\text{Gr}(\mathcal{C})} = 0$ we have 
\begin{equation*}
     =\sigma_{Q_2}|_{\text{Gr}(\mathcal{C}^2)}\sigma_{Q_1F_1^{*}F}|_{\triangle}+\sigma_{Q_2}|_{\triangle}\sigma_{Q_1F_1^{*}F}|_{\text{Gr}(\mathcal{C}^2)} = 0
\end{equation*}

There are two contributions to $\text{Gr}(\mathcal{C}^3)$: $\mathcal{C}^2\circ\mathcal{C}$ and $\mathcal{C}\circ\mathcal{C}^2$, so \eqref{req4} gives,
 \begin{equation*}
      \sigma_{Q_2Q_1F^{*}F}|_{\text{Gr}(\mathcal{C}^3)} = \sigma_{Q_2Q_1F^{*}F}|_{\text{Gr}(\mathcal{C}^2)\circ\text{Gr}(\mathcal{C})}+\sigma_{Q_2Q_1F^{*}F}|_{\text{Gr}(\mathcal{C})\circ\text{Gr}(\mathcal{C}^2)}
\end{equation*}
\begin{equation*}
    =\sigma_{Q_2}|_{\text{Gr}(\mathcal{C})}\sigma_{Q_1F_1^{*}F}|_{\triangle}+\sigma_{Q_2}|_{\triangle}\sigma_{Q_1F_1^{*}F}|_{\text{Gr}(\mathcal{C})}
\end{equation*}
As we already have that $\sigma_{Q_1F_1^{*}F}|_{\text{Gr}(\mathcal{C})} = 0$ and $\sigma_{Q_2}|_{\text{Gr}(\mathcal{C})} = 0$ this reduces to $0$ and we are left with the previous three equations in our system. So we now solve
\begin{equation*}
    \sigma_{Q_2}|_{\triangle}\sigma_{Q_1F_1^{*}F}|_{\triangle} = \omega^{2m}
\end{equation*}
\begin{equation*}
     \sigma_{Q_2}|_{\text{Gr}(\mathcal{C}^2)}\sigma_{Q_1F_1^{*}F}|_{\triangle}+\sigma_{Q_2}|_{\triangle}\sigma_{Q_1F_1^{*}F}|_{\text{Gr}(\mathcal{C}^2)} = 0
\end{equation*}
for $\sigma_{Q_2}|_{\triangle}$ and $\sigma_{Q_2}|_{\text{Gr}(\mathcal{C}^2)}$.

\begin{equation*}
    \begin{pmatrix}
        \sigma_{Q_1F_1^{*}F}|_{\triangle} & 0 \\
        \sigma_{Q_1F_1^{*}F}|_{\text{Gr}(\mathcal{C}^2)} & \sigma_{Q_1F_1^{*}F}|_{\triangle}
    \end{pmatrix}
    \begin{pmatrix}
        \sigma_{Q_2}|_{\triangle}\\
        \sigma_{Q_2}|_{\text{Gr}(\mathcal{C}^2)}
    \end{pmatrix}
    =\begin{pmatrix}
        \omega^{2m} \\
        0
    \end{pmatrix}
\end{equation*}
which can be solved, once again using cramer's rule.

At the $i^{th}$ iteration we wish find $Q_i$ of order 0 from the same class as $Q_{i-1}\dots Q_1 F_1^{*}F$ such that $Q_iQ_{i-1}\dots Q_1$ is elliptic on $\triangle \cup \text{Gr}(\mathcal{C}^{2^i})$ and has order $2m-1$ on $\text{Gr}(\mathcal{C}^{j})$ for $1\leq j\leq 2^i$. So we must solve the system,

\begin{equation*}
    \sigma_{Q_iQ_{i-1}\dots Q_1F_1^*F}|_{\triangle} = \omega^{2m}
\end{equation*}
\begin{equation*}
    \sigma_{Q_iQ_{i-1}\dots Q_1F_1^*F}|_{\text{Gr}(\mathcal{C}^{j})} = 0
\end{equation*}
for $1\leq j< 2^i$.

Since $ \sigma_{Q_{i-1}\dots Q_1F_1^*F}|_{\text{Gr}(\mathcal{C}^j)} = 0$ and $\sigma_{Q_i}|_{\text{Gr}(\mathcal{C}^{j})} = 0$ we have to solve the two-dimensional system,

\begin{equation*}
    \sigma_{Q_iQ_{i-1}\dots Q_1F_1^*F}|_{\triangle} = \omega^{2m}
\end{equation*}
\begin{equation*}
    \sigma_{Q_iQ_{i-1}\dots Q_1F_1^*F}|_{\text{Gr}(\mathcal{C}^{2^{i-1}})} = 0
\end{equation*}
Solving this system then for the amplitude of $Q_i$ facilitates the construction of the desired operator at the $i^{th}$ iteration of this process.

\section{Numerical Experiments}\label{numerics}

We will now verify the theoretical results discussed in this article by implementing them numerically in MATLAB. For our purposes here, to simulate radar data we construct an algorithm representing an approximated version of the forward operator given in equation \eqref{forward}. We denote this simplified version of the forward operator, $F$, as $\tilde{F}$. To construct $\tilde{F}$, we ignore the effects of the amplitude in the forward operator by setting $A_i(\mathbf{x}, \omega, r) = 1$ in \eqref{forward}. Furthermore, we also assume that the antenna is fully broadband such that $\omega \in (-\infty, \infty)$. This allows us to write the oscillatory integral in terms of a Dirac delta distribution. Under these assumptions the simplified forward operator is given by
\begin{equation}
		\tilde{F}V(r,t) \propto \su{i=1,2}{}{\itg{}{}{\delta\left(t-\frac{\vert \mathbf{x}-\gamma(\mathbf{r}) \vert }{c_0} - \frac{\vert \mathbf{x} - \mathbf{E}_i \vert}{c_0}\right)V(\mathbf{x})d\mathbf{x}}}\label{simplified}
\end{equation}

The effect of $\tilde F$ on $V$ in \eqref{simplified} is to integrate $V$ over ellipsoidal surfaces with foci located at $\gamma(\mathbf{r})$ and $\mathbf{E}_i$. Points on this ellipsoid satisfy $| \mathbf{x}-\gamma(\mathbf{r}) | + |\mathbf{x} - \mathbf{E}_i| = c_0t$. Applying this to our input scene will provide a simulation for raw radar data, that would be recorded as a result of emitters located at $\mathbf{E}_1$ and $\mathbf{E}_2$, before any image reconstruction is applied.

Following this, we implement a backprojection algorithm that is applied to the synthetic radar data to produce the simulated image. We choose this backprojection operator to be a simplified version of $F_1^{*}$, which we will denote as $\tilde{F}_1^{*}$. We know from our earlier discussions that applying such an operator to our data will result in both a diagonal and a mixed term, which should respectively correspond to true reconstruction and artifacts as a result of crosstalk between the emitters. The assumptions made to simplify $F_1^{*}$ are the same as those used to obtain \eqref{simplified} as an approximation to \eqref{forward}. The resulting image will then be given by

\begin{equation}
    \tilde{F}_1^{*}F_1V(\mathbf{z}) \propto \itg{}{}{\delta\left(t-\frac{\vert \mathbf{z}-\gamma(\mathbf{r}) \vert }{c_0} - \frac{\vert \mathbf{z} - \mathbf{E}_1 \vert}{c_0}\right)\tilde{F}V(r,t)d\mathbf{r}dt}.\label{simplified_image}
\end{equation}

Finally, once we have produced the desired images, we will apply the artifact mitigation techniques from sections \ref{data_acquisition approach} and \ref{Artifact displacement approach}. We will draw comparison between the results obtained when these methods are used and when they are not.

\subsection{Impact of crosstalk}\label{crosstalk_simulations}
We begin by highlighting the effects of crosstalk by comparing images when crosstalk is ignored with those when crosstalk is accounted for.  When ignoring the effects of crosstalk we generate the data as if it were produced by a just one emitter only, and there was no second, interfering emitter. This corresponds to omitting $i = 2$ in \eqref{simplified}, and \eqref{simplified_image} yields only a diagonal term. To include the effects of crosstalk we apply \eqref{simplified} as it is stated, as if there were two emitters. This results in a diagonal and mixed term in  \eqref{simplified_image}, the latter of which models the impact of crosstalk. 

In both cases we let our input scene, $V(\mathbf{x})$, be a spherical Gaussian function centred at the point $\mathbf{x} = (0, 2, 3)$, that is, $V(\mathbf{x}) = e^{-x_1^2 - (x_2-2)^2 -(x_3-3)^2}$. This function decays rapidly without the presence of a singularity, however, our simulations illustrate that our methods for mitigating artifacts are still remarkably effective for such a scatterer. Slices through the resulting three-dimensional image when crosstalk is ignored are shown in figure \ref{fig:no_cross}. As expected, there is concentration at the location of the original scatterer. However, there is some smearing around this point due to the fixed height flight track only providing a limited angular view of the scatterer.

We contrast these results with the images shown in figure \ref{fig:image}, which show similar slices through the three dimensional image where cross talk has now been taken into account. Here, the same concentration at the location of the true scatterer can be seen, however, artifacts are now clearly present also.

\begin{figure}

\begin{minipage}{.5\linewidth}
\centering
\subfloat[]{\label{no_cross_xy}\includegraphics[width=8cm]{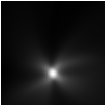}}
\end{minipage}%
\begin{minipage}{.5\linewidth}
\centering
\subfloat[]{\label{no_cross_xz}\includegraphics[width=8cm]{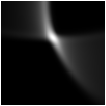}}
\end{minipage}\par\medskip
\centering
\subfloat[]{\label{no_cross_yz}\includegraphics[width=8cm]{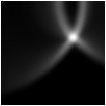}} 

\caption{Slices through the image with no cross talk at $x_3 = 3$, $x_2 = 2$ and $x_1 = 0$ respectively.}
\label{fig:no_cross}
\end{figure}

\begin{figure}

\begin{minipage}{.5\linewidth}
\centering
\subfloat[]{\label{image_xy}\includegraphics[width=8cm]{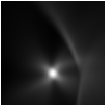}}
\end{minipage}%
\begin{minipage}{.5\linewidth}
\centering
\subfloat[]{\label{image_xz}\includegraphics[width=8cm]{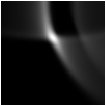}}
\end{minipage}\par\medskip
\centering
\subfloat[]{\label{image_yz}\includegraphics[width=8cm]{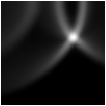}} 

\caption{Slices through the image with cross talk at $x_3 = 3$, $x_2 = 2$ and $x_1 = 0$ respectively.}
\label{fig:image}
\end{figure}

\subsection{Limited data acquisition geometry}
Here, we show how artifacts can be ameliorated using our results from section \ref{data_acquisition approach}. To implement the data acquisition geometry approach we include condition \eqref{subset_condition} in our code and use a cutoff function to mute data satisfying this condition in line with our earlier discussions. We choose the value of $H$ in \eqref{subset_condition} to determine a range in height containing our reflectivity function, within which there will be no artifacts present in the final image.

 When implementing our data acquisition geometry approach to mitigating artifacts we opt for a different kind of reflectivity function than what we used previously. In this simulation we choose the input scene to be a small and narrow, rectangular indicator function located in the region $\{(x_1, x_2, x_3) | -2 < x_1 < 2,\text{ } -1 < x_2 < 1,\text{ } 0 < x_3 < 1/2 \}$. The reason for this choice is due to this method requiring us to omit certain portions of the data before backprojection. As mentioned previously, the fixed height flight track gives us limited views of the scene and results in some degradation in the quality of the image. Omitting even more data here will accentuate this issue. A narrow, rectangular indicator function will mitigate this issue somewhat as, away from the edges and corners, it approximates a distribution with only one singular direction, meaning less views from different angles should be required for the main body of such an object to appear fully formed in the image.

 Figure \ref{without} shows a vertical slice through the image without using this artifact mitigation technique and \ref{with} shows the same slice but with the use of the data acquisition geometry approach to ameliorating the artifacts. It can be seen that using the method has eliminated the artifact as desired. However, as a result of using less data, the reconstruction of the original input scene is slightly more smeared and not quite as defined as when all data is used. 

\begin{figure}

\begin{minipage}{.5\linewidth}
\centering
\subfloat[]{\label{without}\includegraphics[width=8cm]{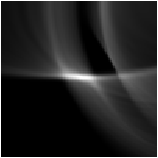}}
\end{minipage}%
\begin{minipage}{.5\linewidth}
\centering
\subfloat[]{\label{with}\includegraphics[width=8cm]{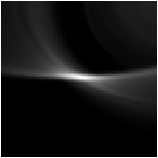}}
\end{minipage}\par\medskip

\caption{Slice through the image before and after artifacts are removed as a result of the data acquisition geometry approach.}
\label{fig:data_acquisition_approach}
\end{figure}

\subsection{Artifact displacement}
Lastly, we numerically implement a single iteration of the artifact displacement method discussed in section \ref{Artifact displacement approach}. This is done for the same setup as in section \ref{crosstalk_simulations} and the images shown in figure \ref{fig:image} will act as references as to what the image looks like before the artifact displacement algorithm has been applied. We proceed with an example calculation which shows how the operator $Q_1$ can be constructed

We recall that $Q_1 \in I^{0}(\triangle)+I^{0}(\text{Gr}(\mathcal{C}))$. Our simulated image is then given by

\begin{equation*}
    I(\mathbf{z}) =\tilde{F}_1^{*}\tilde{F}V =  \tilde{F}_1^{*}(\tilde{F}_1+\tilde{F}_2)V ,
\end{equation*}
\begin{equation*}
     = \tilde{F}_1^{*}\tilde{F}_1V+\tilde{F}_1^{*}F_2V ,
\end{equation*}

By arranging our experiment such that the plane $\pi$ is separated from the region of interest we ensure $\triangle \cap\mathcal{C} = \emptyset$. With this in mind and recalling that each $\tilde{F}_i$ has amplitude equal to 1, we have that

\begin{equation*} \sigma_{\tilde{F}_1^{*}\tilde{F}}|_{\triangle} = \sigma_{\tilde{F}_1^{*}\tilde{F}_1} = 1,
\end{equation*}
\begin{equation*}
     \sigma_{\tilde{F}_1^{*}\tilde{F}}|_{\text{Gr}(\mathcal{C})}=\sigma_{\tilde{F}_1^{*}\tilde{F}_2} = 1,
\end{equation*}
Using this in equations \eqref{amp_diag} and \eqref{amp_mixed},
\begin{equation*}
    \sigma_{Q_1}|_{\triangle} = 1,
\end{equation*}
\begin{equation*}
     \sigma_{Q_1}|_{\text{Gr}(\mathcal{C})}=-1,
\end{equation*}
This means, for the part of $Q_1$ associated with $\triangle$ we apply a term with the same phase as $\tilde{F}_1^{*}\tilde{F}_1$ (a pseudodifferential operator) with amplitude 1. For the part of $Q_1$ associated with $\text{Gr}(\mathcal{C})$ we apply a term with the same phase as $\tilde{F}_1^{*}\tilde{F}_2$ and with amplitude -1. This results in

 \begin{equation*}
     Q_1I(\mathbf{z}) = \tilde{F}_1^{*}\tilde{F}_1(\tilde{F}_1^{*}\tilde{F}_1V + \tilde{F}_1^{*}\tilde{F}_2 V) - \tilde{F}_1^{*}\tilde{F}_2(\tilde{F}_1^{*}\tilde{F}_1V + \tilde{F}_1^{*}\tilde{F}_2 V),
 \end{equation*}
 
  \begin{equation*}
     = (\tilde{F}_1^{*}\tilde{F}_1)(\tilde{F}_1^{*}\tilde{F}_1)V + (\tilde{F}_1^{*}\tilde{F}_1)(\tilde{F}_1^{*}\tilde{F}_2)V - (\tilde{F}_1^{*}\tilde{F}_2)(\tilde{F}_1^{*}\tilde{F}_1)V - (\tilde{F}_1^{*}\tilde{F}_2)(\tilde{F}_1^{*}\tilde{F}_2)V,
 \end{equation*}
 
Cancellation occurs with terms involving $(\tilde{F}_1^{*}\tilde{F}_1)(\tilde{F}_1^{*}\tilde{F}_2)$, which are both associated with $\text{Gr}(\mathcal{C})$, and corresponds to the removal of the original artifacts in the image as desired. Following this we are left with

  \begin{equation}
     I(\mathbf{z}) = (\tilde{F}_1^{*}\tilde{F}_1)(\tilde{F}_1^{*}\tilde{F}_1)V - (\tilde{F}_1^{*}\tilde{F}_2)(\tilde{F}_1^{*}\tilde{F}_2)V.\label{resulting_image}
 \end{equation}

We have that $(\tilde{F}_1^{*}\tilde{F}_1)(\tilde{F}_1^{*}\tilde{F}_1) \in I^{0}(\triangle)$ and the first term in the above will correspond to true reconstruction in the image. We also have that $(\tilde{F}_1^{*}\tilde{F}_2)(\tilde{F}_1^{*}\tilde{F}_2) \in I^{0}(\mathcal{C}^2)$ and the second term in the above will result in the displaced artifacts, that are further away from the ROI. Slices through the image after this algorithm has been applied are shown in figure \ref{fig:displacement}. To handle displaying negative image values that can arise due to the negative amplitude in the second term of \eqref{resulting_image} the absolute value of the values in the final image are taken. Comparing these images to those in figure \ref{fig:image}, it is clear that the crosstalk artifacts have been significantly reduced, as expected.

\begin{figure}

\begin{minipage}{.5\linewidth}
\centering
\subfloat[]{\label{displacement_xy}\includegraphics[width=8cm]{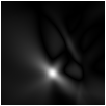}}
\end{minipage}%
\begin{minipage}{.5\linewidth}
\centering
\subfloat[]{\label{displacement_yz}\includegraphics[width=8cm]{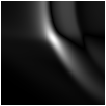}}
\end{minipage}\par\medskip
\centering
\subfloat[]{\label{displacement_xz}\includegraphics[width=8cm]{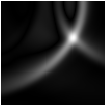}} 

\caption{Slices through the image with after the artifact displacement procedure has been applied.}
\label{fig:displacement}
\end{figure}

\section{Conclusions \& discussion}
In many SAR experiments, crosstalk between emitters has the potential to be a major detriment to the quality of the reconstruction. For example, when imaging in urban environments where there is a high density of signals there is a greater possibility for external, uncontrolled signals to corrupt the radar data. Similar issues can arise in a variety of SAR experiments, such as in passive imaging scenarios where uncontrolled signals are relied upon to produce a reconstruction. To tackle this, we have used a model, taking the form of an FIO, for the radar data that is recorded when scattered waves that emanate from two, always on, sources of illumination are measured. In section \ref{surface_section} we used microlocal analysis to analyse the features that are present in the resulting image when a backprojection algorithm, that assumes the data cannot be separated in the contributions from each emitter, is applied to the data. Through microlocal methods we were able to determine the location of artifacts in the image as a result of a given point scatterer.

We then leveraged our knowledge regarding the locations of crosstalk artifacts in section \ref{mitigation} to develop methods that can mitigate their harmful impact on the quality of reconstructions. In particular, we used two methods to achieve this where the first of which utilised the fact that artifacts are contained on a surface whose shape depends on the geometry of the flight track. By carefully designing the experiment, we showed how it was possible to yield a surface of artifacts that did not intersect the ROI, leaving it artifact free. Following this, we discussed our second artifact mitigation technique in which a further series of FIOs can be applied after backprojection that progressively displace artifacts away from the desired ROI. We discussed the constraints on the experimental design for this method to be applicable and how also how beam forming can be used to extend the range of scenarios in which it can be used. Finally, in section \ref{numerics} we carried out numerical experiments in MATLAB that highlight the effectiveness of our two methods for combatting crosstalk.

Both of the approaches that we have discussed in this article are effective for minimising the harmful effects of crosstalk between signals when imaging a ROI. However, due to each methods own limitations there may be situations where one is preferable over the other. For example, the data acquisition geometry approach requires that only data satisfying certain constraints is used. This creates a limitation on the number of views of the ROI there are and on the angles from which these views are taken, which can cause a degradation in the quality of the image. If possible, the artifact displacement approach will be preferable as it will not restrict the data that can be used in the backprojection. However, it may not always be possible to meet the conditions required for the artifact displacement approach to be viable. This is particularly true in passive imaging experiments where one cannot control the locations of the sources that they are using. In situations such as this where artifacts cannot be displaced, crosstalk artifacts can still be dealt with, with the compromise of using more limited data along in our data acquisition geometry approach.

\section*{Acknowledgments}
This publication has emanated from research conducted with the financial support of Taighde Éireann – Research Ireland under Grant number 18/CRT/6049. For the purpose of Open Access, the author has applied a CC BY public copyright licence to any Author Accepted Manuscript version arising from this submission.

\newpage
\bibliographystyle{unsrt}  
\bibliography{references}  

\end{document}